\documentclass[12pt]{amsart}

\usepackage{amsmath, amssymb}
\usepackage{graphicx}
\usepackage{color}
\usepackage[utf8]{inputenc}
\usepackage{float}
\usepackage{subfigure}

\newtheorem{theorem}{Theorem}[section]
\newtheorem{lemma}[theorem]{Lemma}

\theoremstyle{definition}
\newtheorem{defi}[theorem]{Definition}
\newtheorem{bsp}[theorem]{Example}

\theoremstyle{remark}
\newtheorem{remark}[theorem]{Remark}

\numberwithin{equation}{section}

\newcommand{\rr}{{\mathbb R}}

\newcommand{\nat}{{\mathbb N}}
\newcommand{\ganz}{{\mathbb Z}}
\newcommand{\complex}{{\mathbb C}}
\newcommand{\Exp}{{\mathbb E}}

\newcommand{\sgn}{\operatorname{sgn}}

\allowdisplaybreaks

\begin{document}

\sloppy
\title[Semi-fractional diffusion equations]{Semi-fractional diffusion equations} 
\author{Peter Kern}
\address{Peter Kern, Mathematical Institute, Heinrich-Heine-University D\"usseldorf, Universit\"atsstr. 1, D-40225 D\"usseldorf, Germany}
\email{kern\@@{}hhu.de}

\author{Svenja Lage}
\address{Svenja Lage, Mathematical Institute, Heinrich-Heine-University D\"usseldorf, Universit\"atsstr. 1, D-40225 D\"usseldorf, Germany}
\email{Svenja.Lage@uni-duesseldorf.de}

\author{Mark M. Meerschaert}
\address{Mark M. Meerschaert, Department of Statistics and Probability, Michigan State University,  East Lansing, MI 48824, USA}
\email{mcubed@stt.msu.edu}

\date{\today}

\begin{abstract}
It is well known that certain fractional diffusion equations can be solved by the densities of stable L\'evy motions. In this paper we use the classical semigroup approach for L\'evy processes to define semi-fractional derivatives, which allows us to generalize this statement to semistable L\'evy processes. A Fourier series approach for the periodic part of the corresponding L\'evy exponents enables us to represent semi-fractional derivatives by a Gr\"unwald-Letnikov type formula. We use this formula to calculate semi-fractional derivatives and solutions to semi-fractional diffusion equations numerically. In particular, by means of the Gr\"unwald-Letnikov type formula we provide a numerical algorithm to compute semistable densities.
\end{abstract}


\maketitle

\baselineskip=18pt

\section{Introduction}

Space-fractional diffusion equations are useful to model anomalous diffusions with a faster spreading rate than the classical Brownian motion model predicts \cite{MK}. In this case the behavior is super-diffusive and a spreading rate $t^{1/\alpha}$ with time $t$ for some $\alpha\in(0,2)$ may be modeled by a space-fractional partial differential equation
\begin{equation*}
\frac{\partial}{\partial t}\,p(x,t)=D_1\,\frac{\partial^{\alpha}}{\partial x^{\alpha}}\,p(x,t)+D_2\,\frac{\partial^{\alpha}}{\partial (-x)^{\alpha}}\,p(x,t)
\end{equation*}
for constants $D_1,D_2\in\rr$ with $D_1,D_2\leq0$ if $\alpha\in(0,1)$ or $D_1,D_2\geq0$ if $\alpha\in(1,2)$ and $D_1+D_2\not=0$ in both cases. Together with the initial condition $p(x,0)=\delta_x$ (the Dirac delta distribution) the solution is well-known to consist of densities $x\mapsto p(x,t)$ of a stable L\'evy process $(X_t)_{t\geq0}$; see \cite{Chavez, MMMSik}. Here, $\frac{\partial^{\alpha}}{\partial x^{\alpha}}$ and $\frac{\partial^{\alpha}}{\partial (-x)^{\alpha}}$ are the positive and negative fractional derivatives of order $\alpha\in(0,2)\setminus\{1\}$. We refer to \cite{KST, MMMSik, SKM} for details on fractional derivatives, corresponding fractional partial differential equations and their connection to stochastic processes with heavy tails. 
Numerous applications of fractional diffusions are known from physics, biology, hydrology and other sciences; e.g., see \cite{KST, MMMSik, SKM} and the references cited therein. 

Our goal is to generalize these fractional diffusion models so that the more general class of semistable L\'evy process densities also appears as solutions of certain nonlocal diffusion equations which we will call semi-fractional. The difference between the stable  and the semistable L\'evy processes is that the power law behavior of the tails of the L\'evy measure can additionally be disturbed by log-periodic functions as described below. Such log-periodic perturbations naturally appear in many physical applications \cite{Sor} and also in finance \cite{Sor2}. We further aim to give numerical solutions to semi-fractional diffusion equations by means of a Gr\"unwald-Letnikov type formula similar to \cite{MMMTad} which also provides a method to calculate any semistable probability density numerically. To the best of our knowledge, the only available numerical method to calculate semistable probability densities is given in \cite{Cha} by Laplace inversion techniques for the special case of one-sided semistable distributions.

For some fixed $c>1$ and $\alpha\in(0,2]$ a non-degenerate probability measure $\nu$ is called $(c^{1/\alpha},c)$-semistable if $\nu$ is infinitely divisible and there exists $d\in\mathbb{R}$ such that
\begin{equation}\label{semistable}
 \nu^{\ast c}=(c^{1/\alpha}\nu)\ast\varepsilon_d, 
\end{equation}
where $\nu^{\ast c}$ is the $c$-fold convolution power of $\nu$, well-defined by the L\'evy-Khintchine representation \cite[Theorem 3.1.11]{MMMHPS}, $(c^{1/\alpha}\nu)$ is the image measure of $\nu$ under the dilation $x\mapsto c^{1/\alpha}x$ and $\varepsilon_d$ denotes the Dirac measure concentrated at $d\in\mathbb R$. If \eqref{semistable} holds for every $c>1$ and some $d=d(c)\in\mathbb R$ then $\nu$ is $\alpha$-stable and thus semistability generalizes stability with a discrete scaling property. If $\alpha=2$, then $\nu$ is a normal law and thus stable. We will exclude this case in our considerations and thus for $\alpha\in(0,2)$ the Fourier transform $\widehat\nu(x)=\int_{\rr}e^{ixy}\,d\nu(y)=\exp(\psi(x))$ is given by the log-characteristic function
\begin{equation}\label{logchar}
\psi(x)=iax+\int\limits_{\mathbb{R}\setminus\{0\}}\left(e^{ixy}-1-\frac{ixy}{1+y^2}\right)d\phi(y)
\end{equation}
for some unique $a\in\rr$ and a L\'evy measure $\phi$ determined by
$$
 \phi(r,\infty)=r^{-\alpha}\theta_1(\log r)\quad\text{ and }\quad
 \phi(-\infty,-r)=r^{-\alpha}\theta_2(\log r)
$$
for every $r>0$, where $\theta_1,\theta_2:\mathbb{R}\to[0,\infty)$ are $\log(c^{1/\alpha})$-periodic functions, such that $\theta:=\theta_1+\theta_2$ is strictly positive. 
Due to the fact that $\phi$ is a measure, we know that $r\mapsto r^{-\alpha}\theta(\log r)$ is non-increasing, equivalently $\theta$ fulfills the growth restriction
\begin{equation}\label{growing}
\theta(y+\delta)\leq e^{\alpha\delta}\theta(y)\quad \text{ for every }y\in\rr\text{ and }\delta\geq0
\end{equation}
and \eqref{growing} carries over to $\theta_1$ and $\theta_2$ instead of $\theta$. For more details on semistable distributions we refer the reader to the monographs \cite{MMMHPS} and \cite{Sato}.

The following result is from Theorem 3.17 of \cite{MMMSik}. Let $(X_t)_{t\geq 0}$ be a L\'evy process with Fourier transform $\Exp[\exp(ixX_t)]=\exp(t\,\psi(x))$,
then the family of linear operators $(T_t)_{t\geq0}$ given by
\begin{equation*}
 T_tf(x):=\Exp[f(x-X_t)]\quad\text{ for all }t\geq0\text{ and }x\in\rr
\end{equation*}
determines a $C_0$-semigroup for functions $f$ in $C_0(\mathbb{R})$ with generator
\begin{equation}\label{generator}
Lf(x)=-af'(x)+\int\limits_{\mathbb{R}\setminus\{0\}}\left(f(x-y)-f(x)+\frac{yf'(x)}{1+y^2}\right)d\phi(y),
\end{equation}
where at least functions $f$ with $f,f',f''\in C_0(\mathbb{R})\cap L^1(\mathbb{R})$ belong to the domain of the generator, and we have
\begin{equation}\label{gen2}
\widehat{Lf}(x)=\psi(x)\widehat{f}(x)\quad\text{ for all }x\in\mathbb{R},
\end{equation}
where $\psi$ is the log-characteristic function in \eqref{logchar} and the Fourier transform of a function $f\in L^1(\rr)$ is given by $\widehat f(x)=\int_{\rr}e^{ixy}f(y)\,dy$. 

This is our starting point to introduce semi-fractional derivatives in Section 2 in terms of the generators of semistable semigroups. Using a Fourier series approach for the periodic functions $\theta_1,\theta_2$ we will establish a series representation of the log-characteristic function in Section 3 which sheds some light on semi-fractional derivatives. The Fourier series approach also enables us to give a useful numerical approximation formula to semi-fractional derivatives by a Gr\"unwald-Letnikov type formula in Section 4. Finally, in Section 5 we consider semi-fractional diffusion equations for which the densities of semi-stable L\'evy processes are a valid solution. These densities are then approximated numerically by means of our Gr\"unwald-Letnikov type formula.

\section{Semi-fractional derivatives and diffusion equations}

We introduce the following notation for periodic functions suitably appearing in the log-characteristic function of a semistable law.
\begin{defi}\label{admiss}
Given $c>1$ and $\alpha\in(0,2)$ we call a function $\theta:\rr\to(0,\infty)$ {\it admissable} if $\theta$ is $\log(c^{1/\alpha})$-periodic and fulfills the growth restriction \eqref{growing}. An admissable function $\theta$ as above is called {\it smooth} if $\theta$ is continuous on $\rr$ and piecewise continuously differentiable.
\end{defi}
We will now define semi-fractional derivatives as a direct generalization of classical fractional derivatives by their generator form. For details on the connection of ordinary fractional derivatives and stable distributions we refer to \cite{MMMSik}. Note that the generators of semistable L\'evy processes have been used in \cite{neben} to approximate fractional derivatives as $c\downarrow1$. As in the special case of classical fractional derivatives, we will distinguish between the cases $\alpha\in(0,1)$ and $\alpha\in(1,2)$ and between L\'evy measures concentrated on the positive or negative half-line. We will need the following auxiliary result for $\alpha\in(1,2)$.
\begin{lemma}\label{gamma1}
Given $c>1$, $\alpha\in(1,2)$ and a corresponding admissable function $\theta$ we have
\begin{equation}\label{gammas1}
\int_y^{\infty} x^{-\alpha}\theta(\log x)\,dx=y^{1-\alpha}\gamma(\log y)\quad\text{ for all }y>0,
\end{equation}
where $\gamma$ is the continuous admissable function given by
\begin{equation}\label{gammadef1}
\gamma(x)=e^{(\alpha-1)x}\int_{e^x}^{\infty} y^{-\alpha}\theta(\log y)\,dy\quad\text{ for all }x\in\rr.
\end{equation}
\end{lemma}
\begin{proof}
 Clearly, from \eqref{gammadef1} it follows that $\gamma$ is continuous and $\gamma(x)\in(0,\infty)$ for all $x\in\rr$. Further, $\gamma$ fulfills the growth restriction \eqref{growing}, since the left-hand side in \eqref{gammas1} is strictly decreasing, and by a change of variables $y=z\,c^{1/\alpha}$ we get
 \begin{align*}
\gamma(x+\log c^{1/\alpha}) & =e^{(\alpha-1)x}c^{(\alpha-1)/\alpha}\int_{e^xc^{1/\alpha}}^{\infty} y^{-\alpha}\theta(\log y)\,dy\\
&=e^{(\alpha-1)x}c\int_{e^x}^{\infty} (z\,c^{1/\alpha})^{-\alpha}\theta(\log z+\log c^{1/\alpha})\,dz=\gamma(x)
\end{align*}
for all $x\in\rr$, showing that $\gamma$ is an admissable function. 
\end{proof}

\subsection{Semi-fractional derivatives}

We will first consider $\alpha\in(0,1)$.  Given $c>1$ and $\alpha\in(0,1)$ let $\phi_1$ be the L\'evy measure given by 
\begin{equation}\label{PSI1}
 \phi_1(r,\infty)=r^{-\alpha}\theta(\log r)\quad\text{ and }\quad
 \phi_1(-\infty,-r)=0
\end{equation}
for all $r>0$ and some fixed admissable function $\theta$. Let
$$a_1:=\int\limits_{0+}^\infty\frac{y}{1+y^2}\,d\phi_1(y),$$
which is a positive and finite constant, and write $L_1$ for the corresponding generator in \eqref{generator} with $a=a_1$ and $\phi=\phi_1$. Then the generator takes a simpler form:
\begin{defi} \label{generator1}
With the above assumptions, for a function $f$ belonging to the domain of $L_1$ we define the {\it positive (or left-sided) semi-fractional derivative} of order $\alpha\in(0,1)$ with respect to $c$,  $\theta$ by its generator form
 	\begin{align*}
	\frac{\partial^{\alpha}}{\partial_{c,\theta}x^{\alpha}}\,f(x):=-L_1f(x)=\int_{0+}^{\infty}(f(x)-f(x-y))\,d\phi_1(y).
	\end{align*}
\end{defi}

By reflection of the L\'evy measure $\phi_1$ we get a new L\'evy measure $\phi_2$ on the negative half-line with
$$
 \phi_2(r,\infty)=0\quad\text{ and }\quad
 \phi_2(-\infty,-r)=r^{-\alpha}\theta(\log r)
$$
for all $r>0$. Write $L_2$ for the corresponding generator in \eqref{generator} with $\phi=\phi_2$ and
$$a=a_2:=\int\limits_{-\infty}^{0-}\frac{y}{1+y^2}\,d\phi_2(y)=-\int_{0+}^\infty\frac{y}{1+y^2}\,d\phi_1(y)=-a_1.$$
By a change of variables $y\mapsto -y$ in \eqref{generator}, we may express $L_2$ in terms of the L\'evy measure $\phi_1$ from \eqref{PSI1}.
Then the generator takes a simpler form, similar to the nonlocal operator in Definition \ref{generator1}, but involving the values of $f(y)$ for $y>x$:
\begin{defi} \label{generator2}
With the above assumptions, for a function $f$ belonging to the domain of $L_2$ we define the {\it negative (or right-sided) semi-fractional derivative} of order $\alpha\in(0,1)$ with respect to $c$,  $\theta$ by 
 	\begin{align*}
	\frac{\partial^{\alpha}}{\partial_{c,\theta}(-x)^{\alpha}}\,f(x):=-L_2f(x)=\int_{0+}^{\infty}(f(x)-f(x+y))\,d\phi_1(y).
	\end{align*}
\end{defi}
For functions $f$ with $f,f',f''\in C_0(\mathbb{R})\cap L^1(\mathbb{R})$ we know that the fractional derivative of order $\alpha\in(0,1)$ exists and integration by parts in combination with (\ref{PSI1}) yields the equivalent {\it Caputo form}
\begin{align*}
\frac{\partial^{\alpha}}{\partial_{c,\theta}x^{\alpha}}\,f(x)&=\left[(f(x-y)-f(x))y^{-\alpha}\theta(\log y)\right]_{y=0+}^{\infty}+\int_{0+}^{\infty}f'(x-y)y^{-\alpha}\theta(\log y)\,dy\\
&=\int_{0+}^{\infty}f'(x-y)y^{-\alpha}\theta(\log y)\,dy,
\end{align*}
where the first term vanishes since $\theta$ is bounded, $f, f'\in C_0(\mathbb{R})$ and thus for some constant $C>0$ we have
\begin{align*}
\lim\limits_{y\downarrow 0 }|(f(x-y)-f(x)){y^{-\alpha}}\theta(\log y )|&\leq C\lim\limits_{y\downarrow 0 } |f(x-y)-f(x)|y^{-\alpha}\\
& =C|f'(x)|\lim\limits_{y\downarrow 0 } y^{1-\alpha}=0.
\end{align*}

Analogously, the Caputo form of the negative semi-fractional derivative of order $\alpha\in(0,1)$ is given by
\begin{align*}
\frac{\partial^{\alpha}}{\partial_{c,\theta}(-x)^{\alpha}}\,f(x)&=-\int_{0+}^{\infty}f'(x+y)y^{-\alpha}\theta(\log y)\,dy.
\end{align*}
Note that for constant $\theta\equiv1/\Gamma(1-\alpha)>0$ we get back the usual (positive and negative) Caputo fractional derivatives of order $\alpha\in(0,1)$.

Now we will consider $\alpha\in(1,2)$.  Given $c>1$ and $\alpha\in(1,2)$ let $\phi_1$ be as above and define
$$a_3:=\int\limits_{0+}^\infty\left(\frac{y}{1+y^2}-y\right)\,d\phi_1(y),$$
which is a real constant due to our assumptions. Write $L_3$ for the corresponding generator in \eqref{generator} with $a=a_3$ and $\phi=\phi_1$. Then the generator equation \eqref{generator} takes a simpler form:
\begin{defi} \label{generator3}
With the above assumptions, for a function $f$ belonging to the domain of $L_3$ we define the {\it semi-fractional derivative} of order $\alpha\in(1,2)$ with respect to $c$,  $\theta$ by its generator form
 	\begin{align*}
	\frac{\partial^{\alpha}}{\partial_{c,\theta}x^{\alpha}}\,f(x):=L_3f(x)=\int_{0+}^{\infty}(f(x-y)-f(x)+y\, f'(x))\,d\phi_1(y).
	\end{align*}
\end{defi}

Again, by reflection of the L\'evy measure $\phi_1$ we get the L\'evy measure $\phi_2$ on the negative half-line and we write $L_4$ for the corresponding generator in \eqref{generator} with 
$$a=a_4:=\int\limits_{-\infty}^{0-}\left(\frac{y}{1+y^2}-y\right)\,d\phi_2(y)=-\int\limits_{0+}^{\infty}\left(\frac{y}{1+y^2}-y\right)\,d\phi_1(y)=-a_3.$$
Again, by a change of variables $y\mapsto -y$ in \eqref{generator}, we may express $L_4$ in terms of the L\'evy measure $\phi_1$ from \eqref{PSI1}. With this L\'evy measure and the above shift, we obtain the following generator:
\begin{defi} \label{generator4}
With the above assumptions, for a function $f$ belonging to the domain of $L_4$ we define the {\it negative semi-fractional derivative} of order $\alpha\in(1,2)$ with respect to $c$,  $\theta$ by 
 	\begin{align*}
	\frac{\partial^{\alpha}}{\partial_{c,\theta}(-x)^{\alpha}}\,f(x):=L_4f(x)=\int_{0+}^{\infty}(f(x+y)-f(x)-y\,f'(x))\,d\phi_1(y).
	\end{align*}
\end{defi}
As before, for functions $f$ with $f,f',f''\in C_0(\mathbb{R})\cap L^1(\mathbb{R})$ we know that the (negative) fractional derivative of order $\alpha\in(1,2)$ exists. Integrate by parts twice to obtain the equivalent Caputo forms
\begin{equation}\label{Capsfd}\begin{split}
\frac{\partial^{\alpha}}{\partial_{c,\theta}x^{\alpha}}\,f(x)&=\int_{0+}^{\infty}\left(f'(x)-f'(x-y)\right)y^{-\alpha}\theta(\log y)\,dy\\
& =\int_{0+}^{\infty}f''(x-y)\, y^{1-\alpha}\gamma(\log y)\,dy,
\end{split}\end{equation}
where $\gamma$ is the function from Lemma \ref{gamma1}, using that $f',f''\in C_0(\mathbb{R})$ and the fact that $\gamma$ is bounded, we obtain for some constant $C>0$
\begin{align*}
\lim\limits_{y\downarrow 0} \left|(f(x-y)-f(x))y^{1-\alpha}\gamma(\log y)\right|&\leq C \lim\limits_{y\downarrow 0} \left|(f(x-y)-f(x))\right|y^{1-\alpha}\\
&\leq C|f''(x)|\lim_{y\downarrow0}y^{2-\alpha}=0.
\end{align*}
Analogously we get
\begin{align*}
\frac{\partial^{\alpha}}{\partial_{c,\theta}(-x)^{\alpha}}\,f(x)&=\int_{0+}^{\infty}\left(f'(x+y)-f'(x)\right)y^{-\alpha}\theta(\log y)\,dy\\
& =\int_{0+}^{\infty}f''(x+y)\, y^{1-\alpha}\gamma(\log y)\,dy.
\end{align*}
Note that for constant $\theta\equiv (\alpha-1)/\Gamma(2-\alpha)=-1/\Gamma(1-\alpha)>0$, i.e.\ $\gamma\equiv1/\Gamma(2-\alpha)$, we get back the usual positive and negative) Caputo fractional derivatives of order $\alpha\in(1,2)$.

\subsection{Zolotarev derivative and ballistic scaling}

In our previous considerations we have excluded the case $\alpha=1$ and we now briefly illustrate by an example the technical difficulties that may arise for $\alpha=1$. A prominent example of a semistable distribution is the limit distribution of cumulated gains in a sequence of St.\ Petersburg games, where $c=2$ and $\alpha=1$. The L\'evy measure of the semistable limit distribution was first established in \cite{ML} and is the discrete measure $\phi_1$ given by
$$\phi_1(2^k)=2^{-k}\quad\text{ for all }k\in\ganz.$$
One can easily show that for $r>0$ we have
$$
 \phi_1(r,\infty)=r^{-\alpha}\theta(\log r)\quad\text{ and }\quad
 \phi_1(-\infty,-r)=0,
$$
where $\theta$ is the admissable function
$$\theta(x)=\exp\left(x-\left\lfloor\frac{x}{\log2}\right\rfloor\log 2\right)\quad\text{ for all }x\in\rr.$$
We further obtain
\begin{align*}
a_1 &=\int\limits_{0+}^\infty\frac{y}{1+y^2}\,d\phi_1(y)=\sum_{k=-\infty}^\infty\frac1{1+4^k}=\infty\\
\intertext{and}
a_3 &=\int\limits_{0+}^\infty\left(\frac{y}{1+y^2}-y\right)\,d\phi_1(y)=-\int\limits_{0+}^\infty\frac{y^3}{1+y^2}\,d\phi_1(y)=-\sum_{k=-\infty}^\infty\frac1{1+4^{-k}}=-\infty.
\end{align*}
Hence, both Definitions \ref{generator1} and \ref{generator3} fail to give a suitable semi-fractional derivative in this special case with $\alpha=1$. To overcome this problem, recently in \cite{KLM} the Zolotarev fractional derivative was introduced. In the ballistic case $\alpha=1$ a direct generalization of the Zolotarev fractional derivative in \cite {KLM} to our semistable situation is the following. As in Section 3.2.3 of \cite{HPBA} we define 
$$a=a_5:=\int_{0+}^\infty\left(\frac{y}{1+y^2}-\sin y\right)\,d\phi_1(y)$$
in the L\'evy-Khintchine representation \eqref{logchar} of the log-characteristic function $\psi$ with L\'evy measure $\phi_1$ as above and write $L_{\mathcal Z}^+$ for the corresponding generator in \eqref{generator}:
\begin{defi}\label{generator5}
With the above assumptions, for a function $f$ belonging to the domain of $L_{\mathcal Z}^+$ we define the {\it positive Zolotarev semi-fractional derivative} of order $\alpha=1$ with respect to $c$,  $\theta$ by its generator form
\begin{align*}
\frac{\partial_{\mathcal Z}}{\partial_{c,\theta}x}\,f(x):=L_{\mathcal Z}^+ f(x)=\int_{0+}^{\infty}(f(x-y)-f(x)+f'(x)\sin y)\,d\phi_1(y).
\end{align*}
\end{defi}
Note that for functions $f$ with $f,f',f''\in C_0(\mathbb{R})\cap L^1(\mathbb{R})$ by the mean value theorem we observe for some $\xi\in[x-y,x]$ and $\eta\in[\xi,x]$
\begin{align*}
f(x-y)-f(x)+f'(x)\sin y & =-y\,f'(\xi)+f'(x)\sin y\\
& =f'(x)(\sin y -y)+y\,(f'(x)-f'(\xi))\\
& =f'(x)(\sin y -y)+y\,(x-\xi)f''(\eta)
\end{align*}
showing that as $y\downarrow0$
$$\left| f(x-y)-f(x)+f'(x)\sin y \right|\leq \|f'\|_\infty |\sin y-y|+y^2\|f''\|_\infty=O(y^2).$$
Hence the Zolotarev semi-fractional derivative of order $\alpha=1$ exists for such functions and integration by parts yields the {\it Caputo form}
$$\frac{\partial_{\mathcal Z}}{\partial_{c,\theta}x}\,f(x)=\int_{0+}^{\infty}\left(f'(x)\cos y-f'(x-y)\right)y^{-1}\theta(\log y)\,dy.$$
Note that for constant $\theta\equiv 2/\pi$ we get back the Zolotarev fractional derivative of order $\alpha=1$ in \cite{KLM}.

Again, by reflection of the L\'evy measure $\phi_1$ we get the L\'evy measure $\phi_2$ on the negative half-line and we write $L_{\mathcal Z}^-$ for the corresponding generator in \eqref{generator} with 
$$a:=\int_{-\infty}^{0-}\left(\frac{y}{1+y^2}-\sin y\right)\,d\phi_2(y)=-\int_{0+}^{\infty}\left(\frac{y}{1+y^2}-\sin y\right)\,d\phi_1(y)=-a_5,$$
then, by a change of variables $y\mapsto -y$ in \eqref{generator}, this generator takes the simpler form:

\begin{defi}\label{generator6}
With the above assumptions, for a function $f$ belonging to the domain of $L_{\mathcal Z}^-$ we define the {\it negative Zolotarev semi-fractional derivative} of order $\alpha=1$ with respect to $c$,  $\theta$ by its generator form
\begin{align*}
\frac{\partial_{\mathcal Z}}{\partial_{c,\theta}(-x)}\,f(x):=L_{\mathcal Z}^- f(x)=\int_{0+}^{\infty}(f(x+y)-f(x)-f'(x)\sin y)\,d\phi_1(y).
\end{align*}
\end{defi}
As before, for functions $f$ with $f,f',f''\in C_0(\mathbb{R})\cap L^1(\mathbb{R})$ the negative Zolotarev semi-fractional derivative of order $\alpha=1$ exists and integration by parts yields the {\it Caputo form}
$$\frac{\partial_{\mathcal Z}}{\partial_{c,\theta}(-x)}\,f(x)=\int_{0+}^{\infty}\left(f'(x+y)-f'(x)\cos y\right)y^{-1}\theta(\log y)\,dy.$$

\section{semistable log-characteristic functions}

An alternative, equivalent definition of the semi-fractional derivative using the L\'evy-Khintchine representation \eqref{logchar} together with \eqref{gen2} is the following. Given $c>1$, $\alpha\in(0,2)\setminus\{1\}$ and a corresponding admissable function $\theta$, with our choice of the shift $a_1$ for $\alpha\in(0,1)$ and $a_3$ for $\alpha\in(1,2)$, we may define $\partial^\alpha f/\partial_{c,\theta}x^\alpha$ as the function with Fourier transform
\begin{equation}\label{FTsfd}
\widehat{\frac{\partial^\alpha f}{\partial_{c,\theta}x^\alpha}}(x)=\begin{cases}
\displaystyle -\widehat{L_1f}(x)=-\int_{0+}^\infty(e^{ixy}-1)\,d\phi_1(y)\,\widehat f(x) & \text{ if }\alpha\in(0,1),\\[2ex]
\displaystyle \widehat{L_3f}(x)=\int_{0+}^\infty(e^{ixy}-1-ixy)\,d\phi_1(y)\,\widehat f(x) & \text{ if }\alpha\in(1,2),
\end{cases}
\end{equation}
for suitable functions $f$. Analogously, with the shift $a_2$ for $\alpha\in(0,1)$ and $a_4$ for $\alpha\in(1,2)$, we can define $\partial^\alpha f/\partial_{c,\theta}(-x)^\alpha$ as the function with Fourier transform
\begin{equation}\label{FTnsfd}
\widehat{\frac{\partial^\alpha f}{\partial_{c,\theta}(-x)^\alpha}}(x)=\begin{cases}
\displaystyle -\widehat{L_2f}(x)=-\int_{0+}^\infty(e^{-ixy}-1)\,d\phi_1(y)\,\widehat f(x) & \text{ if }\alpha\in(0,1),\\[2ex]
\displaystyle \widehat{L_4f}(x)=\int_{0+}^\infty(e^{-ixy}-1+ixy)\,d\phi_1(y)\,\widehat f(x) & \text{ if }\alpha\in(1,2).
\end{cases}
\end{equation}
for suitable functions $f$, as well as the positive Zolotarev semi-fractional derivative for $\alpha=1$ (using the shift $a_5$) as the function with Fourier transform
\begin{align}\label{FTsfZd}
\widehat{\frac{\partial_{\mathcal{Z}} f}{\partial_{c,\theta}x^\alpha}} (x)=\widehat{L_{\mathcal Z}^+f}(x)=\int_{0+}^\infty(e^{ixy}-1-ix\sin(y))\,d\phi_1(y)\,\widehat f(x),
\end{align}
respectively
\begin{align}\label{FTnsfZd}
 \widehat{\frac{\partial_{\mathcal{Z}} f}{\partial_{c,\theta}(-x)^\alpha}} (x)=\widehat{L_{\mathcal Z}^-f}(x)=\int_{0+}^\infty(e^{-ixy}-1+ix\sin(y))\,d\phi_1(y)\,\widehat f(x) 
\end{align}
for the negative Zolotarev semi-fractional derivative.

Our aim is to reduce \eqref{FTsfd} -- \eqref{FTnsfZd} to simpler forms. Since the right-hand sides in \eqref{FTsfd} -- \eqref{FTnsfZd} coincide with $\mp\psi(x)\,\widehat f(x)$ for the log-characteristic functions $\psi$ without drift part, corresponding to the generators $L_1$ to $L_4$ in Definitions  \ref{generator1}--\ref{generator4} and $L_{\mathcal{Z}}^{\pm}$ in Definitions \ref{generator5} and \ref{generator6}, we will now derive a series representation of $\psi$ depending on the Fourier coefficients of the periodic function $\theta$, provided $\theta$ is a smooth admissable function. This will directly provide us with a series representation of the Fourier transform of semi-fractional derivatives and will also enable us to give a formula of Gr\"unwald-Letnikov type for the semi-fractional derivatives in Section 4.
Given $c>1$ and $\alpha\in(0,2)$, let $\theta$ be a smooth admissable function as in Definition \ref{admiss}. In this case for all $x\in\rr$ we have the Fourier series representation
\begin{equation}\label{Fseries}
\theta(x)=\sum_{k=-\infty}^{\infty} c_k e^{ik\tilde c x}\quad\text{ with }\quad \tilde{c}:=\frac{2\pi\alpha}{\log c}
\end{equation}
and the Fourier coefficients $(c_k)_{k\in\mathbb{Z}}\subseteq\mathbb{C}$ fulfill
$$\overline{c_{-k}}=c_k\quad\text{ and }\quad |c_k|\leq\frac{C}{k^2}\quad\text{ for all }k\in\ganz\setminus\{0\}\text{ and some }C>0;$$
e.g., see \cite{Fourier1} for details.

\subsection{Log-characteristic function for $\alpha\not=1$}

\begin{theorem}\label{psiserrep}
Given $c>1$, $\alpha\in(0,2)\setminus\{1\}$ and a corresponding smooth admissable function $\theta$ with Fourier series representation \eqref{Fseries}, let $\psi$ denote the log-characteristic function corresponding to the generator $L_1$ in Definition \ref{generator1} for $\alpha\in(0,1)$, respectively the generator $L_3$ in Definition \ref{generator3} for $\alpha\in(1,2)$. Then we have
\begin{equation}\label{Psi1}
\psi(x)=-\sum\limits_{k=-\infty}^{\infty} c_k\,\Gamma(ik\tilde{c}-\alpha+1)\,(-ix)^{\alpha-ik\tilde{c}}\quad\text{ for all }x\in\rr.
\end{equation}
\end{theorem}
To prove Theorem \ref{psiserrep} we will need the following refinement of Lemma \ref{gamma1} for the case $\alpha\in(1,2)$.
\begin{lemma}\label{gamma}
Given $c>1$, $\alpha\in(1,2)$ and a corresponding smooth admissable function $\theta$ with Fourier series representation \eqref{Fseries}, the admissable function $\gamma$ from Lemma \ref{gamma1} is continuously differentiable with Fourier series representation
\begin{equation}\label{gammadef}
\gamma(x)=\sum\limits_{k=-\infty}^{\infty}\frac{c_k}{\alpha-1-ik\tilde c}\, e^{ik\tilde{c}x}\quad\text{ for all }x\in\rr.
\end{equation}
\end{lemma}
\begin{proof}
 For any $y>0$ we obtain by dominated convergence
 \begin{align*}
 \int_y^{\infty} x^{-\alpha}\theta(\log x)dx &=\int_y^{\infty} x^{-\alpha}\sum\limits_{k=-\infty}^{\infty} c_k e^{ik\tilde{c}\log(x)}dx=\int_y^{\infty}\sum\limits_{k=-\infty}^{\infty}c_k x^{-\alpha+ik\tilde{c}}\,dx\\ 
&=\sum\limits_{k=-\infty}^{\infty} \int_y^{\infty} c_k\,x^{-\alpha+ik\tilde{c}} dx
=-\sum\limits_{k=-\infty}^{\infty} c_k\, \frac{y^{-\alpha+ik\tilde{c}+1}}{ik\tilde{c}-\alpha+1}\\
	&=y^{-\alpha+1}\sum\limits_{k=-\infty}^{\infty} \frac{c_k}{\alpha-1-ik\tilde c}\, e^{ik\tilde{c}\log(y)}\end{align*}
showing \eqref{gammadef} by Lemma \ref{gamma1}. Since the admissable function $\theta$ is smooth we have $|c_k|\leq C\,k^{-2}$ for all $k\in\ganz\setminus\{0\}$ and some $C>0$. Thus we get
\begin{align*}
	\left|\frac{c_k}{\alpha-1-ik\tilde{c}}\right|
	=\frac{|c_k|}{\sqrt{(\alpha-1)^2+k^2\tilde{c}^2}}
	\leq \frac{|c_k|}{|k|\tilde{c}}\leq\frac{C}{|k|^3\tilde c},
	\end{align*}
which according to Theorem 2.6 in \cite{Fourier1} leads to continuous differentiability of $\gamma$. 
\end{proof}
\begin{proof}[Proof of Theorem \ref{psiserrep}]
We first consider the case $\alpha\in(0,1)$. For the log-characteristic function $\psi$ corresponding to the generator $L_1$ we obtain by dominated convergence and integration by parts
\begin{align*}
\psi(x)&=\int_{0+}^\infty(e^{ixy}-1)\,d\phi_1(y)
=\int_{0+}^{\infty}\lim\limits_{n\to\infty}\left(e^{\left(ix-\frac{1}{n}\right)y}-1\right)\,d\phi_1(y)\\
&=\lim\limits_{n\to\infty}\int_{0{+}}^{\infty}\left(e^{\left(ix-\frac{1}{n}\right)y}-1\right)\,d\phi_1(y)\\
&=\lim\limits_{n\to\infty}\int_{0+}^{\infty} \left(ix-\frac{1}{n}\right)e^{\left(ix-\frac{1}{n}\right)y}y^{-\alpha}\theta(\log y)\,dy
\end{align*}
Since $\theta$ is a smooth admissable function with Fourier series \eqref{Fseries} we have for some constant $C_1>0$
\begin{align*}
\left|e^{\left(ix-\frac{1}{n}\right)y}y^{-\alpha}\theta(\log y)\right| & \leq\sum_{k=-\infty}^\infty\left|e^{\left(ix-\frac{1}{n}\right)y}y^{-\alpha}c_k\, y^{ik\tilde c}\right|\\
& \leq \sum_{k=-\infty}^\infty e^{-y/n}y^{-\alpha}|c_k|\leq C_1e^{-y/n}y^{-\alpha}
\end{align*}
and thus by dominated convergence we get
\begin{equation}\label{philim}\begin{split}
\psi(x) & =\lim\limits_{n\to\infty} \left(ix-\frac{1}{n}\right)\sum\limits_{k=-\infty}^{\infty} c_k\int_{0+}^{\infty}e^{\left(ix-\frac{1}{n}\right)y}y^{-\alpha+ik\tilde{c}}\,dy\\
& =\lim\limits_{n\to\infty} \left(ix-\frac{1}{n}\right)\sum\limits_{k=-\infty}^{\infty} c_k\left(\frac{1}{n}-ix\right)^{-ik\tilde{c}+\alpha-1}\Gamma(ik\tilde{c}-\alpha+1),
\end{split}\end{equation}
where the last equality follows as on page 144 in \cite{Integral}, since $\Re(1-\alpha+ik\tilde c)=1-\alpha>0$. We further obtain for $k\in\ganz\setminus\{0\}$ using $|\Gamma(ik\tilde c-\alpha+1)|\leq C_2|k|^{-\alpha+1/2}e^{-|k|\tilde c\pi/2}$ as on page 20 of \cite{Stirling}
\begin{align*}
& \left|c_k\left(\frac{1}{n}-ix\right)^{-ik\tilde{c}+\alpha-1}\Gamma(ik\tilde{c}-\alpha+1)\right|\\
& \quad\leq |c_k| (x^2+n^{-2})^{\alpha/2}\exp\left(k\tilde c\,\arg\left(\frac1n-ix\right)\right)|\Gamma(ik\tilde{c}-\alpha+1)|\\
& \quad\leq C_2|c_k| (x^2+1)^{\alpha/2}e^{|k|\tilde c\pi/2}|k|^{-\alpha+1/2}e^{-|k|\tilde c\pi/2}\\
& \quad\leq C_2C(x^2+1)^{\alpha/2}|k|^{-\alpha-3/2},
\end{align*}
where the last inequality follows again by the smoothness assumption for $\theta$, since $|c_k|\leq C|k|^{-2}$ for some $C>0$. Thus the series in \eqref{philim} is absolutely convergent and by dominated convergence we get
\begin{align*}
\psi(x)&=-\sum\limits_{k=-\infty}^{\infty}c_k \lim\limits_{n\to\infty}\left(\frac{1}{n}-ix\right)^{\alpha-ik\tilde{c}}\Gamma(ik\tilde{c}-\alpha+1)\\
&=-\sum\limits_{k=-\infty}^{\infty}c_k (-ix)^{\alpha-ik\tilde{c}}\,\Gamma(ik\tilde{c}-\alpha+1),
\end{align*}
showing \eqref{Psi1} in case $\alpha\in(0,1)$.

Now we consider the case $\alpha\in(1,2)$. Since the proof is similar to the case $\alpha\in(0,1)$ we only list the main steps and skip the technical details. For the log-characteristic function $\psi$ without drift part, corresponding to the generator $L_3$, we obtain by dominated convergence and integration by parts
\begin{align*}
\psi(x)&=\int_{0+}^\infty(e^{ixy}-1-ixy)\,d\phi_1(y)\\
&=\int_{0+}^{\infty}\lim\limits_{n\to\infty}\left(e^{\left(ix-\frac{1}{n}\right)y}-1-\left(ix-\frac1n\right)y\right)\,d\phi_1(y)\\
&=\lim\limits_{n\to\infty}\int_{0+}^{\infty} \left(ix-\frac{1}{n}\right)\left(e^{\left(ix-\frac{1}{n}\right)y}-1\right)y^{-\alpha}\theta(\log y)\,dy.
\end{align*}
Let $\mu$ be the measure on $(0,\infty)$ with Lebesgue density $y\mapsto y^{-\alpha}\theta(\log y)$, then by Lemma \ref{gamma} we have $\mu(y,\infty)=y^{1-\alpha}\gamma(\log y)$ for a continuously differentiable admissable function $\gamma$. Thus, as in the case $\alpha\in(0,1)$ we get with integration by parts
\begin{align*}
\psi(x) & =\lim\limits_{n\to\infty} \left(ix-\frac{1}{n}\right)\int_{0+}^{\infty}\left(e^{\left(ix-\frac{1}{n}\right)y}-1\right)y^{-\alpha}\theta(\log y)\,d\mu(y)\\
& =\lim\limits_{n\to\infty}\left(ix-\frac{1}{n}\right)^2\int_{0+}^{\infty}e^{\left(ix-\frac{1}{n}\right)y}y^{1-\alpha}\gamma(\log y)\,dy.
\end{align*}
Since $\gamma$ is a smooth admissable function with Fourier series \eqref{gammadef} and $\alpha-1\in(0,1)$, we can proceed as above to obtain
\begin{align*}
\psi(x) &=\lim\limits_{n\to\infty} \left(ix-\frac{1}{n}\right)^2\sum\limits_{k=-\infty}^{\infty} \frac{c_k}{\alpha-1-ik\tilde c}\left(\frac{1}{n}-ix\right)^{-ik\tilde{c}+\alpha-2}\Gamma(ik\tilde{c}-\alpha+2)\\
&=-\sum\limits_{k=-\infty}^{\infty}c_k \lim\limits_{n\to\infty}\left(\frac{1}{n}-ix\right)^{\alpha-ik\tilde{c}}\Gamma(ik\tilde{c}-\alpha+1)\\
&=-\sum\limits_{k=-\infty}^{\infty}c_k (-ix)^{\alpha-ik\tilde{c}}\,\Gamma(ik\tilde{c}-\alpha+1)
\end{align*}
concluding the proof.
\end{proof}
\begin{remark}\label{FTreps}
According to \eqref{Psi1}, we can also represent the log-characteristic function as 
$$\psi(x)=-|x|^\alpha h(x)\quad\text{ for all }x\in\rr,$$
where
$$h(x)=\sum\limits_{k=-\infty}^{\infty} c_k\,\Gamma(ik\tilde{c}-\alpha+1)\,(-i \sgn(x))^{\alpha-ik\tilde{c}}|x|^{-ik\tilde c}$$
showing that $h$ is bounded  and $x\mapsto h(e^x)$ and $x\mapsto h(-e^x)$ are continuous and $\log(c^{1/\alpha})$-periodic functions. In particular we have $\psi(0)=-\lim_{x\to0}|x|^\alpha h(\log|x|)=0$. Together with \eqref{FTsfd} we may also write for $x\in\rr$
\begin{equation}\label{FTsfdalt}
\widehat{\frac{\partial^\alpha f}{\partial_{c,\theta}x^\alpha}}(x)=\sum\limits_{k=-\infty}^{\infty} \omega_k\,(-ix)^{\alpha-ik\tilde{c}}\widehat f(x),
\end{equation}
where
$$\omega_k:=\begin{cases}c_k\Gamma(ik\tilde c-\alpha+1) & \text{ if }\alpha\in(0,1),\\
-c_k\Gamma(ik\tilde c-\alpha+1) & \text{ if }\alpha\in(1,2).
\end{cases}$$
Equation \eqref{FTnsfd} shows that for the Fourier transform of the negative semi-fractional derivative, due to reflection, we simply have to replace $x$ by $-x$ in \eqref{Psi1}, so that
\begin{equation}\label{FTnsfdalt}
\widehat{\frac{\partial^\alpha f}{\partial_{c,\theta}{(-x)}^\alpha}}(x)=\sum\limits_{k=-\infty}^{\infty} \omega_k\,(ix)^{\alpha-ik\tilde{c}}\widehat f(x).
\end{equation}
\end{remark}

\subsection{Log-characteristic function for $\alpha=1$}

\begin{theorem}\label{psiserrepZ}
Given $c>1$, $\alpha=1$ and a corresponding smooth admissable function $\theta$ with Fourier series representation \eqref{Fseries}, let $\psi$ denote the log-characteristic function corresponding to the generator $L_{\mathcal Z}^+$ in Definition \ref{generator5} with $\widehat{L_{\mathcal Z}^+f}=\psi\cdot\widehat f$ as in \eqref{FTsfZd}. Then we have
\begin{equation}\label{Psi1Z}
\psi(x)=-\sum\limits_{k\in\ganz\setminus\{0\}} c_k\,\Gamma(ik\tilde{c})\left((-ix)^{1-ik\tilde{c}}+ ix\,\cos\left(\frac{\pi}{2}ik\tilde{c}\right)\right)-c_0\,ix\log(-ix)
\end{equation}
for all $x\in\mathbb{R}$ where we define $0\cdot\log(0):=0$. 
\end{theorem}

\begin{proof}
Since the statement is obviously true for $x=0$, we only consider $x\neq 0$. 
 The log-characteristic function $\psi$ corresponding to the generator $L_{\mathcal Z}^+$ is given by
 \begin{align*}
 \psi(x)&=\int_{\mathbb{R}}\left(e^{ixy}-1-ix\sin(y)\right)d\phi_1(y)\\
 &=\lim\limits_{n\to\infty} \int_{0+}^{\infty}\left(e^{\left(ix-\frac{1}{n}\right)y^{1+1/n}}-1-\left(ix-\frac{1}{n}\right)\sin(y^{1+1/n})\right)d\phi_1(y)\\
 &=\lim\limits_{n\to\infty}\left(1+\frac{1}{n}\right)\left(ix-\frac{1}{n}\right)\int_{0+}^{\infty}y^{1/n}\left(e^{\left(ix-\frac{1}{n}\right)y^{1+1/n}}-\cos(y^{1+1/n})\right)y^{-1}\theta(\log\; y)dy
 \end{align*}
 for every $x\in\mathbb{R}\setminus\{0\}$ using dominated convergence and integration by parts. To solve the first part of the integral above, for all $n\in\nat$ we obtain by dominated convergence and a change of variables $z=y^{1+1/n}$
 \begin{align*}
 I_1^n(x):&=\left(1+\frac{1}{n}\right)\int_{0+}^{\infty}y^{\frac{1}{n}}e^{\left(ix-\frac{1}{n}\right)y^{1+\frac{1}{n}}}y^{-1}\theta(\log\; y)\,dy\\
 & =\sum\limits_{k=-\infty}^{\infty}c_k\int_{0+}^{\infty}e^{\left(ix-\frac{1}{n}\right)z}z^{\frac1{n+1}+\frac{n}{n+1}ik\tilde c-1}\,dz\\
 &=\sum\limits_{k=-\infty}^{\infty}c_k\left(\frac{1}{n}-ix\right)^{-\frac{1}{n+1}-ik\tilde{c}\frac{n}{n+1}}\Gamma\left(\frac{1}{n+1}+ik\tilde{c}\frac{n}{n+1}\right),
 \end{align*}
where the last equality follows from page 144 in \cite{Integral}, since $\Re(\frac1{n+1}+\frac{n}{n+1}ik\tilde c)=\frac{1}{n+1}>0$. Similarly, for the remaining part of the integral we obtain
  \begin{align*}
 I_2^n:&=\left(1+\frac{1}{n}\right)\int_{0+}^{\infty}y^{\frac{1}{n}}\cos(y^{1+\frac{1}{n}})y^{-1}\theta(\log\; y)\,dy\\
 &= \sum\limits_{k=-\infty}^{\infty}c_k\int_{0+}^{\infty}\cos(z)\,z^{\frac1{n+1}+\frac{n}{n+1}ik\tilde c-1}\,dz.
 \end{align*}
Since $\Re\left(\frac1{n+1}+\frac{n}{n+1}ik\tilde c)\right)=\frac1{n+1}\in(0,1)$, we get
\begin{align*}
I_2^n&=\sum\limits_{k=-\infty}^{\infty}c_k\Gamma\left(\frac{1}{n+1}+\frac{n}{n+1}ik\tilde{c}\right)\cos\left(\frac{\pi}{2}\left(\frac{1}{n+1}+\frac{n}{n+1}ik\tilde{c}\right)\right)
\end{align*}
according to page 319 in \cite{Integral}. Combining these two results, the log-characteristic function reads as follows
\begin{align*}
 \psi(x)=&\lim\limits_{n\to\infty}\left(ix-\frac{1}{n}\right)(I_1^n(x)-I_2^n)\\
 =&\sum\limits_{k=-\infty}^{\infty} c_k\lim\limits_{n\to\infty} \left(ix-\frac{1}{n}\right)\Gamma\left(\frac{1}{n+1}+\frac{n}{n+1}ik\tilde{c}\right)\\
 &\left(\left(\frac{1}{n}-ix\right)^{-\frac{1}{n+1}-ik\tilde{c}\frac{n}{n+1}}-\cos\left(\frac{\pi}{2}\left(\frac{1}{n+1}+\frac{n}{n+1}ik\tilde{c}\right)\right)\right)
\end{align*}
where the last equation follows with dominated convergence. To compute the limit for $k=0$ we make use of the fact that
\begin{align*}
 \cos\left(\frac{\pi}{2(n+1)}\right)&=1-O(n^{-2})\quad\text{and}\quad
 \left(\frac{1}{n}-ix\right)^{-\frac{1}{n+1}}=1-\frac{\log(-ix)}{n}+O(n^{-2}).
\end{align*}
Hence, since $\Gamma(\frac1{n+1})=(n+1)\Gamma(1+\frac1{n+1})$ we get
$$\Gamma\left(\frac{1}{n+1}\right)\left(\left(\frac{1}{n}-ix\right)^{-\frac{1}{n+1}}-\cos\left(\frac{\pi}{2(n+1)}\right)\right)\to -\log(-ix)$$
as $n\to\infty$. Since the gamma function is continuous on $\mathbb{C}\setminus\{-1,-2,\ldots\}$, alltogether we obtain
\begin{align*}
 \psi(x)=&-\sum\limits_{k\in\ganz\setminus\{0\}} c_k \Gamma\left(ik\tilde{c}\right)\left(\left(-ix\right)^{1-ik\tilde{c}}+ix\,\cos\left(\frac{\pi}{2}ik\tilde{c}\right)\right)\\
 &+c_0\lim\limits_{n\to\infty} \left(ix-\frac{1}{n}\right)\Gamma\left(\frac{1}{n+1}\right)\left(\left(\frac{1}{n}-ix\right)^{-\frac{1}{n+1}}-\cos\left(\frac{\pi}{2(n+1)}\right)\right)\\
 &=-\sum\limits_{k\in\ganz\setminus\{0\}} c_k \Gamma\left(ik\tilde{c}\right)\left(\left(-ix\right)^{1-ik\tilde{c}}+ix\,\cos\left(\frac{\pi}{2}ik\tilde{c}\right)\right)-c_0\,ix\log(-ix).
\end{align*}
for every $x\in\mathbb{R}\setminus\{0\}$ concluding the proof. 
\end{proof}

\begin{remark}
 Analogously to Remark \ref{FTreps} we may define
 \begin{align}\label{omega1Z}
  \omega_{k,1}:=\begin{cases} -c_k\Gamma(ik\tilde{c}) &\textrm{ if } k\neq 0\\
                 0 &\textrm{ if } k= 0\\
                \end{cases}
 \end{align}
 and
 \begin{align}\label{omega2Z}
  \omega_{k,2}:=\begin{cases} c_k\Gamma(ik\tilde{c})\cos\left(\frac{\pi}{2}ik\tilde{c}\right) &\textrm{ if } k\neq 0\\
                 0 &\textrm{ if } k= 0\\
                \end{cases}
 \end{align}
so that for suitable functions $f$ by \eqref{FTsfZd} we can write
 \begin{align}\label{FTsfZdalt}
  \widehat{\frac{\partial_{\mathcal{Z}} f}{\partial_{c,\theta}x^\alpha}} (x)=\left(\sum\limits_{k\in\ganz\setminus\{0\}} \omega_{k,1}\left(-ix\right)^{1-ik\tilde{c}}+\omega_{k,2}(-ix) +c_0(-ix)\log(-ix)\right)\widehat{f}(x).
 \end{align}
A comparison of \eqref{FTsfZd} and \eqref{FTnsfZd} shows that for the Fourier transform of the negative Zolotarev semi-fractional derivative, due to reflection, we simply have to replace $x$ by $-x$ in \eqref{Psi1Z}, so that
 \begin{align}\label{FTnsfZdalt}
  \widehat{\frac{\partial_{\mathcal{Z}} f}{\partial_{c,\theta}(-x)^\alpha}} (x)=\left(\sum\limits_{k\in\ganz\setminus\{0\}} \omega_{k,1}\left(ix\right)^{1-ik\tilde{c}}+\omega_{k,2}\,ix +c_0\,ix\log(ix)\right)\widehat{f}(x).
 \end{align}
\end{remark}

\subsection{Continuity at $\alpha=1$ of log-characteristic functions} 

Let $(\alpha_n)_{n\in\nat}$ be a sequence in $(0,2)\setminus\{1\}$ with $\alpha_n\to1$. Define $\tilde c_n:=2\pi\alpha_n/\log c$ and assume that
\begin{equation}\label{thetan}
\theta_n(x):=\sum_{k\in\ganz}c_k\,e^{ik\tilde c_nx},\quad x\in\rr,
\end{equation}
is a sequence of smooth admissable functions with respect to $c>1$ and $\alpha_n\not=1$. For the corresponding log-characteristic functions we write
\begin{equation}\label{psin}
\psi_n(x)=-\sum\limits_{k=-\infty}^{\infty} c_k\,\Gamma(ik\tilde c_n-\alpha_n+1)\,(-ix)^{\alpha_n-ik\tilde c_n}
\end{equation}
according to Theorem \ref{psiserrep}. Clearly, we have for all $x\in\rr$
\begin{equation}\label{thetanconv}
\theta_n(x)\to\sum_{k\in\ganz}c_k\,e^{ik\tilde cx}=:\theta(x),
\end{equation}
where $\tilde c=2\pi/\log c$ and $\theta$ is a smooth admissable function with respect to $c>1$ and $\alpha=1$. Hence, according to Theorem \ref{psiserrepZ} the corresponding log-characteristic function is given by
\begin{equation}\label{psiZ}
\psi_{\mathcal Z}(x)=-\sum\limits_{k\in\ganz\setminus\{0\}} c_k\,\Gamma(ik\tilde{c})\left((-ix)^{1-ik\tilde{c}}+ ix\,\cos\left(\frac{\pi}{2}ik\tilde{c}\right)\right)-c_0\,ix\log(-ix).
\end{equation}
We will now show that appropriate shifts of $\psi_n$ converge to $\psi_{\mathcal Z}$ and thus get a certain continuity result as $\alpha\to1$ for the log-characteristic functions, which transfers to weak convergence of the corresponding semistable  distributions by L\'evy's continuity theorem. 
\begin{theorem}
For all $n\in\nat$ define
\begin{equation}\label{shifts}
d_n:=\sum_{k\in\ganz}c_k\,\Gamma(ik\tilde c_n-\alpha_n+1)\,\cos\left(\frac{\pi}{2}(ik\tilde c_n-\alpha_n+1)\right)
\end{equation}
then for all $x\in\rr$ as $n\to\infty$ we have
\begin{equation}\label{psiconv}
\psi_n(x)+ix\,d_n\to\psi_{\mathcal Z}(x).
\end{equation}
Further, the shifts $d_n$ are representable as 
\begin{equation}\label{shiftrep}
d_n=\begin{cases}
\displaystyle\int_{0+}^\infty\cos(x)\,x^{-\alpha_n}\theta_n(\log x)\,dx & \text{ if }\alpha_n\in(0,1),\\[2ex]
\displaystyle\int_{0+}^\infty(\cos(x)-1)\,x^{-\alpha_n}\theta_n(\log x)\,dx & \text{ if }\alpha_n\in(1,2).
\end{cases}
\end{equation}
\end{theorem}
\begin{proof}
First note that $\tilde c_n\to\tilde c$ as $\alpha_n\to1$ and thus by dominated convergence we get for all $x\in\rr$
$$\sum\limits_{k\in\ganz\setminus\{0\}} c_k\,\Gamma(ik\tilde c_n-\alpha_n+1)\,(-ix)^{\alpha_n-ik\tilde c_n}\to\sum\limits_{k\in\ganz\setminus\{0\}} c_k\,\Gamma(ik\tilde c)\,(-ix)^{1-ik\tilde c}.$$
Hence to prove \eqref{psiconv} it remains to show
$$c_0\Gamma(1-\alpha_n)(-ix)^{\alpha_n}+ix\,d_n\to ix\sum_{k\in\ganz\setminus\{0\}}c_k\,\Gamma(ik\tilde c)\,\cos\left(\frac{\pi}{2}ik\tilde c)\right)+c_0ix\log(-ix),$$
which according to \eqref{shifts} reduces to
$$\Gamma(1-\alpha_n)\left((-ix)^{\alpha_n}+ix\cos\left(\frac{\pi}{2}(1-\alpha_n)\right)\right)\to ix\log(-ix).$$
For the latter we observe
\begin{align*}
& \Gamma(1-\alpha_n)\left((-ix)^{\alpha_n}+ix\cos\left(\frac{\pi}{2}(1-\alpha_n)\right)\right)\\
& \quad=ix\,\Gamma(2-\alpha_n)\,\frac{-(-ix)^{\alpha_n-1}+\cos\left(\frac{\pi}{2}(1-\alpha_n)\right)}{1-\alpha_n}\\
& \quad=ix\,\Gamma(2-\alpha_n)\left(\frac{e^{(\alpha_n-1)\log(-ix)}-1}{\alpha_n-1}+\frac{\cos\left(\frac{\pi}{2}(1-\alpha_n)\right)-1}{1-\alpha_n}\right)\\
& \quad\to ix\,\Gamma(1)\left(\left.\frac{d}{dt}\,e^{t\log(-ix)}\right|_{t=0}+\left.\frac{d}{dt}\,\cos\left(\frac{\pi}{2}\,t\right)\right|_{t=0}\right)=ix\log(-ix),
\end{align*}
concluding the proof of \eqref{psiconv}. Further, for $\alpha_n\in(0,1)$ an application of (21) on page 319 of \cite{Integral} to \eqref{shifts} and dominated convergence shows
\begin{align*}
d_n & =\sum_{k\in\ganz}c_k\,\Gamma(ik\tilde c_n-\alpha_n+1)\,\cos\left(\frac{\pi}{2}(ik\tilde c_n-\alpha_n+1)\right)\\
& =\sum_{k\in\ganz}c_k\int_{0+}^\infty\cos(x)\,x^{-\alpha_n+ik\tilde c_n}\,dx=\int_{0+}^\infty\cos(x)\,x^{-\alpha_n}\sum_{k\in\ganz}c_k\,e^{ik\tilde c_n\log x}\,dx\\
& =\int_{0+}^\infty\cos(x)\,x^{-\alpha_n}\theta_n(\log x)\,dx
\end{align*}
and a similar calculation in case $\alpha_n\in(1,2)$, applying (12) on page 348 of \cite{Integral} instead, concludes the proof of \eqref{shiftrep}.
\end{proof}
\begin{remark}
Note that the representation \eqref{shiftrep} directly shows that for functions $f$ with $f,f',f''\in C_0(\mathbb{R})\cap L^1(\mathbb{R})$ we get convergence of the shifted Caputo forms
$$\frac{\partial^{\alpha_n}}{\partial_{c,\theta_n}x^{\alpha_n}}\,f(x)+d_nf'(x)\to\frac{\partial_{\mathcal Z}}{\partial_{c,\theta}x}\,f(x)$$
if dominated convergence can be applied here. This can even be shown without the smoothness assumption on the admissable functions $\theta_n$, we only need that the Fourier coefficients $(c_n)_{n\in\ganz}$ are absolutely summable, to ensure that $\theta_n(x)\to\theta(x)$, and that dominated convergence can be applied to the shifted Caputo forms. Using \eqref{FTsfd}, \eqref{FTsfZd} and Plancherel's Theorem as in Theorem 3.11 of \cite{KLM} for the special case of constant $\theta_n=\theta\equiv c_0$, this further shows $L^2$-convergence
$$\left\|\frac{\partial^{\alpha_n}f}{\partial_{c,\theta_n}x^{\alpha_n}}+d_nf'-\frac{\partial_{\mathcal Z}f}{\partial_{c,\theta}x}\right\|_2\to0$$
of the corresponding shifted semi-fractional derivatives for suitable functions $f$ with $\widehat f\in L^2(\rr)$ and $\widehat{Lf}\in L^2(\rr)$ for all generators $L$ corresponding to the semi-fractional derivatives $\partial^{\alpha_n}/\partial_{c,\theta_n}x^{\alpha_n}$ and $\partial_{\mathcal Z}/\partial_{c,\theta}x$, if again dominated convergence can be applied.
\end{remark}

\section{semi-fractional Gr\"unwald-Letnikov type formula}

The results of the last section provide an infinitesimal approach to semi-fractional derivatives. This can be applied to approximate semi-fractional derivatives numerically. It will also enable us to give a numerical algorithm for the solution of certain semi-fractional diffusion equations in Section 5. Similar to the assumption above, for given $c>1$ and $\alpha\in(0,2)\setminus\{1\}$ let $\theta$ be a fixed smooth admissible function with Fourier coefficients
$(c_k)_{k\in\mathbb{Z}}$ and recall from Remark \ref{FTreps}
\begin{equation*}
\omega_k=\begin{cases}c_k\Gamma(ik\tilde c-\alpha+1) & \text{ if }\alpha\in(0,1),\\
-c_k\Gamma(ik\tilde c-\alpha+1) & \text{ if }\alpha\in(1,2).
\end{cases}
\end{equation*} 
\begin{defi}
Let $\theta$ be a smooth admissable function with respect to $c>1$ and $\alpha\in(0,2)\setminus\{1\}$. For every $h>0$ and a bounded function $f$ we define the {\it Gr\"unwald-Letnikov semi-fractional difference} in $x\in\rr$ by
\begin{equation}\label{GLdiff}
{}_{h}\Delta^\alpha_{c,\theta}f(x):=\sum\limits_{k=-\infty}^{\infty}\omega_k\,h^{ik\tilde{c}-\alpha}\sum\limits_{j=0}^{\infty}\binom{\alpha-ik\tilde{c}}{j}(-1)^jf(x-jh),
\end{equation}
which is well-defined and real-valued due to the following result.
\end{defi}	
\begin{lemma}\label{abscon}
For every bounded function $f$ and $x\in\rr$ the double series in \eqref{GLdiff} is absolutely convergent and ${}_{h}\Delta^\alpha_{c,\theta}f(x)\in\rr$.
\end{lemma}
\begin{proof}
First note that since $\theta$ is a smooth admissable function the Fourier coefficients $(c_k)_{k\in\ganz}$ are absolutely summable and fulfill $\overline{c_{-k}}=c_k$. Thus also  $(\omega_k)_{k\in\ganz}$ is absolutely summable due to $|\Gamma(ik\tilde c-\alpha+1)|\leq|\Gamma(1-\alpha)|$ for all $k\in\ganz$. Further, by Theorem VI.1 in \cite{Flojolet} we have $|\binom{z}{j}|\leq C\cdot j^{-1-\Re(z)}$ for all $z\in\complex$, $j\in\nat$ and some $C>0$ and hence with $M:=\|f\|_\infty$ we get
\begin{equation*}
\left|{}_{h}\Delta^\alpha_{c,\theta}f(x)\right|\leq\sum\limits_{k=-\infty}^{\infty}|\omega_k|\,h^{-\alpha}M\bigg(1+C\sum\limits_{j=1}^{\infty}j^{-(1+\alpha)}\bigg)<\infty.
\end{equation*}
Using 
$$\binom{\alpha-ik\tilde{c}}{j}(-1)^j=\binom{ik\tilde{c}-\alpha+j-1}{j}=\frac{\Gamma(ik\tilde c-\alpha+j)}{j!\,\Gamma(ik\tilde c-\alpha)}$$
we may rewrite
\begin{align*}
{}_{h}\Delta^\alpha_{c,\theta}f(x) & =\pm h^{-\alpha}\sum\limits_{k=-\infty}^{\infty}c_k\Gamma(ik\tilde c-\alpha+1)\sum\limits_{j=0}^{\infty}\frac{\Gamma(ik\tilde c-\alpha+j)}{j!\,\Gamma(ik\tilde c-\alpha)}\,f(x-jh)e^{ik\tilde c\log h}\\
& =\pm h^{-\alpha}\sum\limits_{k=-\infty}^{\infty}\bigg(c_k(ik\tilde c-\alpha)\sum\limits_{j=0}^{\infty}\frac{\Gamma(ik\tilde c-\alpha+j)}{j!}\,f(x-jh)\bigg)e^{ik\tilde c\log h}\\
& =:\pm h^{-\alpha}\sum\limits_{k=-\infty}^{\infty}a_k\,e^{ik\tilde c\log h}
\end{align*}
to see that ${}_{h}\Delta^\alpha_{c,\theta}f(x)\in\rr$ iff $\overline{a_{-k}}=a_k$ for all $k\in\ganz$. Using $\overline{\Gamma(z)}=\Gamma(\bar z)$ for all $z\in\complex\setminus\{0,-1,-1,\ldots\}$ we get
\begin{align*}
\overline{a_{-k}} & =\overline{c_{-k}}\,(\overline{-ik\tilde c-\alpha})\sum\limits_{j=0}^{\infty}\frac{\overline{\Gamma(-ik\tilde c-\alpha+j)}}{j!}\,f(x-jh)\\
& =c_k(ik\tilde c-\alpha)\sum\limits_{j=0}^{\infty}\frac{\Gamma(ik\tilde c-\alpha+j)}{j!}\,f(x-jh)=a_k
\end{align*}
concluding the proof.
\end{proof}
\begin{lemma}\label{konv}
Let $f\in L^1(\mathbb{R})$ be bounded and let $\theta$ be a smooth admissable function with respect to $c>1$ and $\alpha\in(0,2)\setminus\{1\}$. Then for every $x\in\mathbb{R}$ as $h\downarrow 0$ we have
\begin{align*}
\widehat{{}_{h}\Delta^{\alpha}_{c,\theta}f}(x)\to \sum\limits_{k=-\infty}^{\infty} \omega_k (-ix)^{\alpha-ik\tilde{c}}\widehat{f}(x).
\end{align*}
\end{lemma}
\begin{proof}
 For fixed $h>0$ and every $x\in\mathbb{R}$ we obtain using dominated convergence
 \begin{align*}
\widehat{{}_{h}\Delta^{\alpha}_{c,\theta}f}(x) &=\int\limits_{\mathbb{R}}e^{ixy} {}_{h}\Delta^{\alpha}_{c,\theta}f(y)\,dy\\
& =\int\limits_{\mathbb{R}}e^{ixy}\sum\limits_{k=-\infty}^{\infty}\sum\limits_{j=0}^{\infty}\omega_k h^{ik\tilde{c}-\alpha}\binom{\alpha-ik\tilde{c}}{j}(-1)^jf(y-jh)\,dy\\
&=\sum\limits_{k=-\infty}^{\infty}\omega_kh^{ik\tilde{c}-\alpha}\sum\limits_{j=0}^{\infty}\binom{\alpha-ik\tilde{c}}{j}(-1)^j\int\limits_{\mathbb{R}}e^{ixy}f(y-jh)\,dy\\
&=\sum\limits_{k=-\infty}^{\infty}\omega_kh^{ik\tilde{c}-\alpha}\sum\limits_{j=0}^{\infty}\binom{\alpha-ik\tilde{c}}{j}(-1)^je^{ixjh}\widehat{f}(x)\\
	&=\sum\limits_{k=-\infty}^{\infty}\omega_kh^{ik\tilde{c}-\alpha}\widehat{f}(x)\sum\limits_{j=0}^{\infty} \binom{\alpha-ik\tilde{c}}{j}(-e^{ixh})^j.
	\end{align*}
Since $\Re(\alpha-ik\tilde{c})>0$ it follows that
\begin{align*}
\sum\limits_{j=0}^{\infty} \binom{\alpha-ik\tilde{c}}{j}(-e^{ixh})^j=(1-e^{ixh})^{\alpha-ik\tilde{c}},
\end{align*}
e.g., see \cite[p.397-398]{binom}. Thus by dominated convergence we get
\begin{align*}
\lim_{h\downarrow0}\widehat{{}_{h}\Delta^{\alpha}_{c,\theta}f}(x)& =\lim_{h\downarrow0}\sum\limits_{k=-\infty}^{\infty}\omega_k\left(\frac{1-e^{ixh}}{h}\right)^{\alpha-ik\tilde{c}}\widehat{f}(x)\\
&=\sum\limits_{k=-\infty}^{\infty}\omega_k\lim\limits_{h\downarrow 0}\left(\frac{1-e^{ixh}}{h}\right)^{\alpha-ik\tilde{c}}\widehat{f}(x)=\sum\limits_{k=-\infty}^{\infty} \omega_k (-ix)^{\alpha-ik\tilde{c}}\widehat{f}(x)
\end{align*}
concluding the proof.
\end{proof}
Combining \eqref{FTsfdalt} and Lemma \ref{konv}, Fourier inversion directly yields a {\it Gr\"unwald-Letnikov type formula} for the semi-fractional derivative.
\begin{theorem}\label{GLsfd}
Let $\theta$ be a smooth admissable function with respect to $c>1$ and $\alpha\in(0,2)\setminus\{1\}$. Further, let $f\in L^1(\mathbb{R})$ be a bounded function such that all derivatives of $f$ up to an integer order $n>\alpha+1$ exist and $f^{(n)}\in L^1(\mathbb{R})$. Then for almost every $x\in\mathbb{R}$ we have
\begin{equation*}
\frac{\partial^{\alpha}}{\partial_{c,\theta}x^{\alpha}}\,f(x)=\lim\limits_{h\downarrow 0}{}_{h}\Delta^{\alpha}_{c,\theta}f(x)=\lim\limits_{h\downarrow 0}\sum\limits_{k=-\infty}^{\infty}\omega_k\,h^{ik\tilde{c}-\alpha}\sum\limits_{j=0}^{\infty}\binom{\alpha-ik\tilde{c}}{j}(-1)^jf(x-jh).
\end{equation*}
\end{theorem}
\begin{remark}\label{nsf}
Analogously, by \eqref{FTnsfdalt} and the same steps of proof as in Lemma \ref{konv}, with the same conditions on $f$ as in Theorem \ref{GLsfd} we get a Gr\"unwald-Letnikov type formula for the negative semi-fractional derivative 
\begin{equation*}
\frac{\partial^{\alpha}}{\partial_{c,\theta}(-x)^{\alpha}}\,f(x)=\lim\limits_{h\downarrow 0}\sum\limits_{k=-\infty}^{\infty}\omega_k\,h^{ik\tilde{c}-\alpha}\sum\limits_{j=0}^{\infty}\binom{\alpha-ik\tilde{c}}{j}(-1)^jf(x+jh).
\end{equation*}
for almost every $x\in\mathbb{R}$.
\end{remark}
\begin{bsp}\label{ex1}
 Let $f(x)=\exp(-x^2)$ then $f,f',f''\in C_0(\mathbb{R})\cap L^1(\mathbb{R})$ so that the generator and Caputo forms of semi-fractional derivatives are equivalent. For fixed $\alpha\in(0,2)\setminus\{1\}$, we define the $2\pi$-periodic function 
  \begin{align*}
  \theta(x)=\begin{cases} \displaystyle\frac{\alpha\sin(x)}{6\,\Gamma(1-\alpha)} +\frac{1}{\Gamma(1-\alpha)} &\textrm{ if } \alpha\in(0,1),\\[2ex]
              \displaystyle\frac{\alpha\sin(x)}{6\,\Gamma(1-\alpha)} -\frac{1}{\Gamma(1-\alpha)} &\textrm{ if } \alpha\in(1,2).
             \end{cases}
 \end{align*}
 Thus, eliminating the first term we will receive the ordinary fractional derivative of order $\alpha$.
 Then $\theta$ is a smooth admissable function and according to Theorem \ref{GLsfd} the Gr\"unwald-Letnikov formula approximates the semi-fractional derivative of $f$ of order $\alpha$ with respect to $c:=e^{2\pi\alpha}$ and $\theta$.
 \begin{figure}
\subfigure[{(Semi-)fractional derivative of order $\alpha=1.5$ from Example \ref{ex1} in the interval $[-5,5]$.}]{\includegraphics[width=0.49\textwidth]{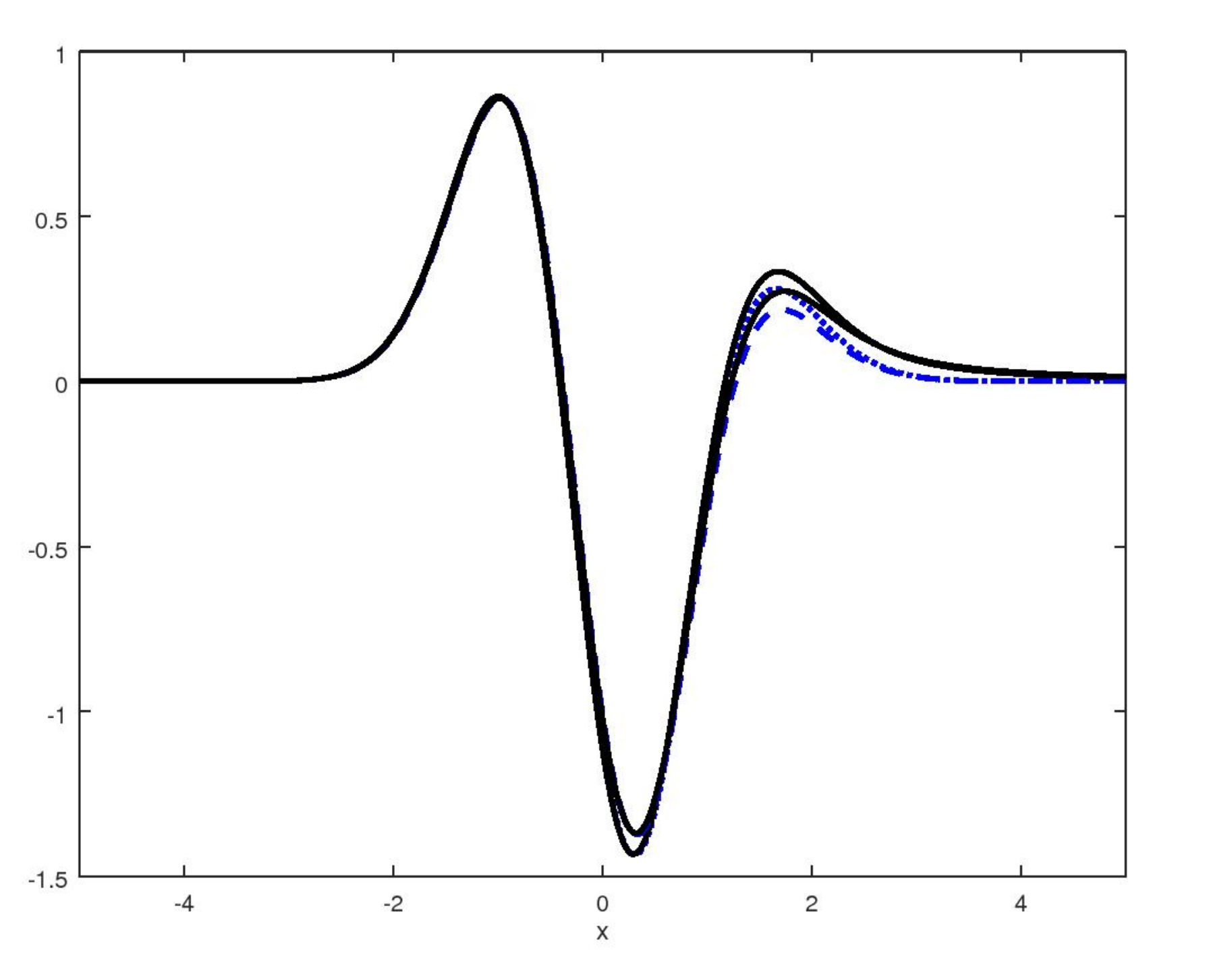}}
  \subfigure[{Zoom of (a) to the interval $[0,0.5]$.}]{\includegraphics[width=0.50\textwidth]{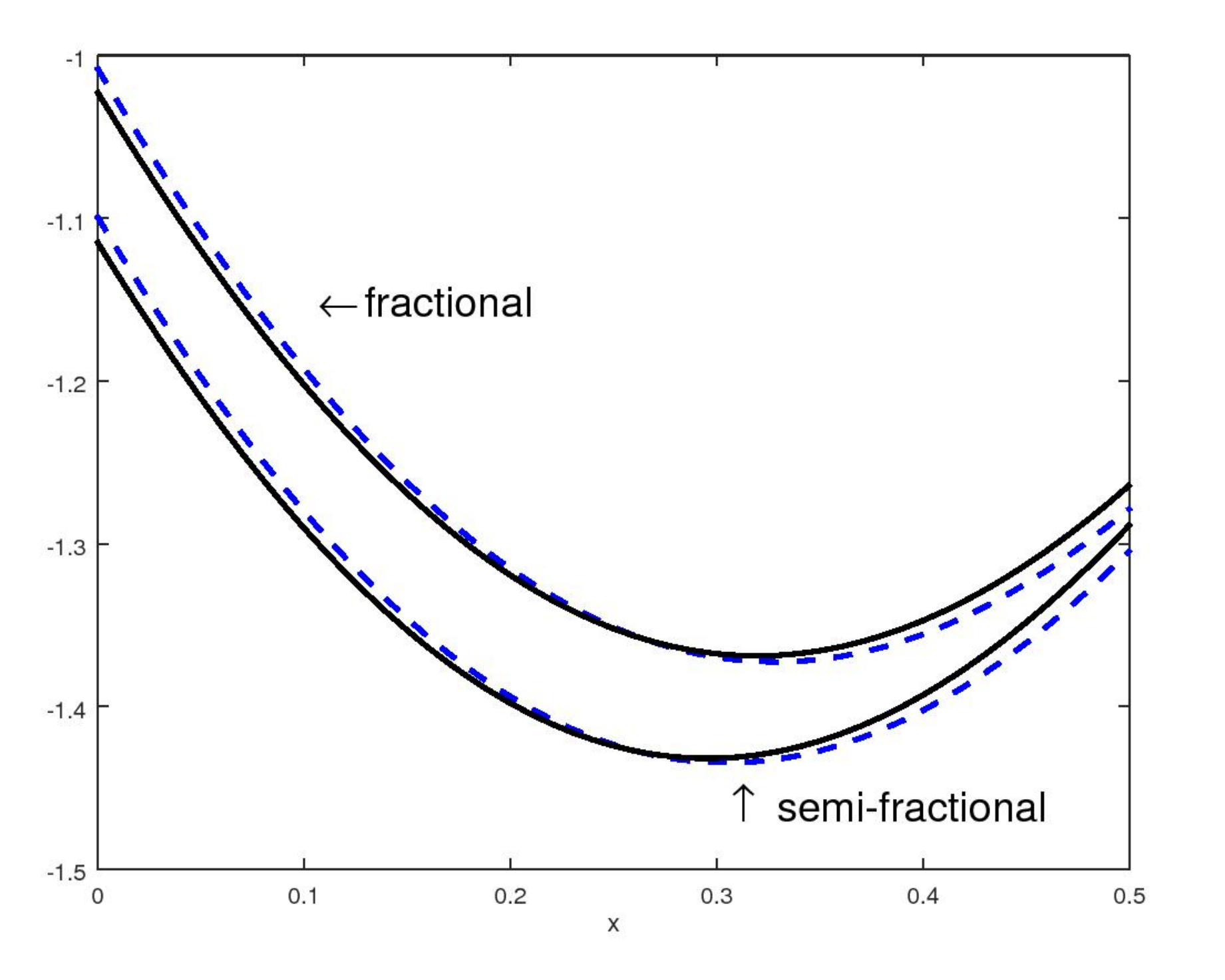}}
\caption{Comparison of numerically evaluated Caputo forms (solid lines) and Gr\"unwald-Letnikov approximations (dashed lines) for fractional and semi-fractional derivatives.}\label{Fig1}
\end{figure}
For $\alpha=1.5$, the numerically evaluated Caputo form \eqref{Capsfd} of the fractional (cancel the sine part in the definition of $\theta$) and the semi-fractional derivative of Definition \ref{generator1} on the intervals $[-5,5]$ and $[0,0.5]$ are shown in Figure \ref{Fig1} together with the corresponding Gr\"unwald-Letnikov approximation of the semi-fractional derivative. For the numerical approximation of the Caputo forms we used the function \textit{quadcc} in GNU Octave \cite{Octave}, which uses adaptive Clenshaw-Curtis rules to calculate the integral. For all computations, we used a step size of $h=0.01$ and for each point of interest, we truncated the inner sum to $j\leq200$ in the Gr\"unwald-Letnikov approximation \eqref{GLdiff}. 
 \begin{figure}
  \includegraphics[width=0.75\textwidth]{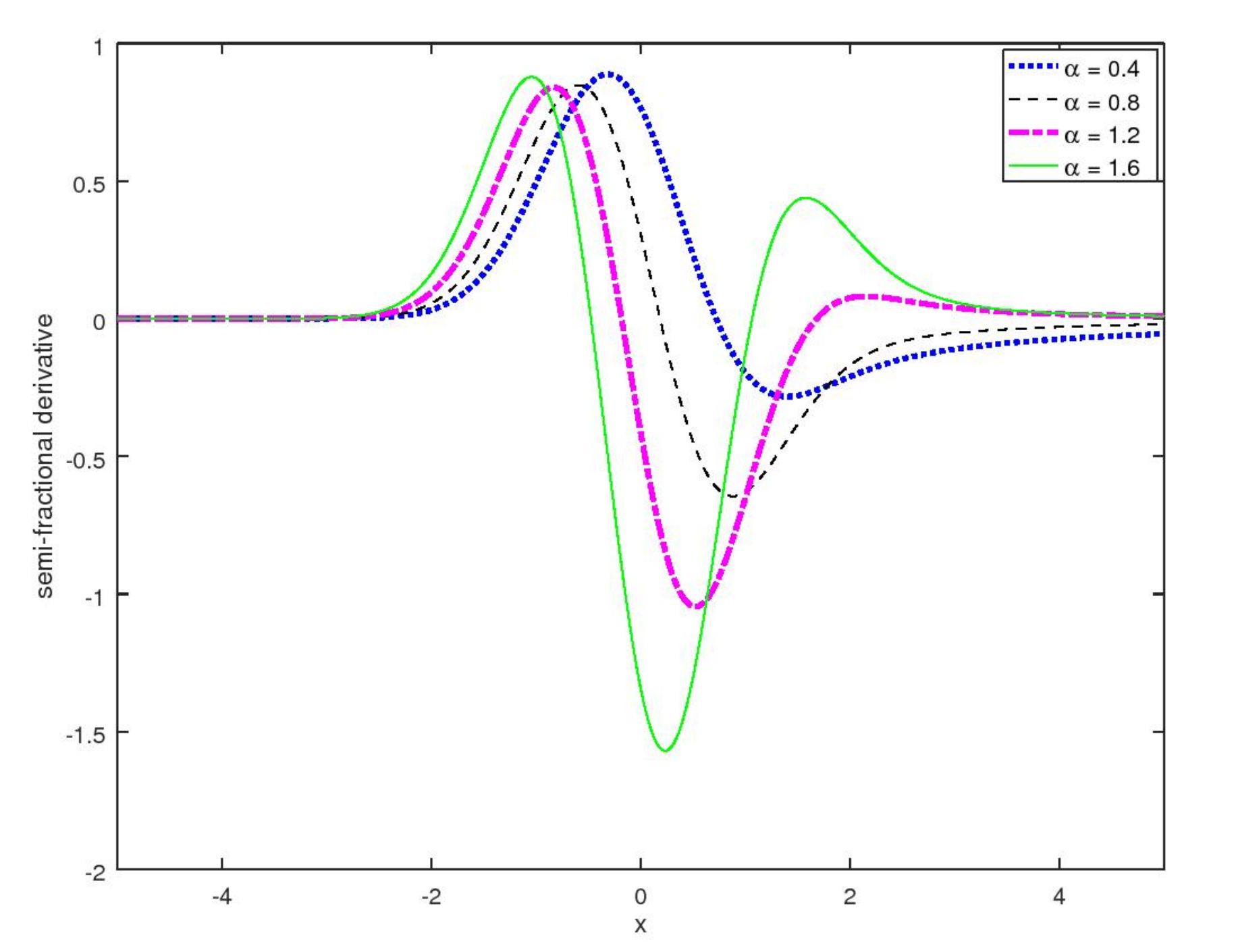}
\caption{Semi-fractional derivative from Example \ref{ex1} for different orders $\alpha$ on the interval $[-5,5]$.}
\label{Fig2}
\end{figure} 
To study the influence of the parameter $\alpha$ on this derivative, we varied $\alpha$ between $0.4$ and $1.6$ in Figure \ref{Fig2}. The derivatives shown are the Caputo forms of the semi-fractional derivatives numerically evaluated by the same {\it quadcc} method as above.
\end{bsp}

Finally, we want to answer the question of Grünwald-Letnikov type formulas for the Zolotarev semi-fractional derivative of order $\alpha=1$. Let again $c>1$ be fixed and let $\theta$ be a smooth admissable function with Fourier
coefficients $(c_k)_{k\in\mathbb{Z}}$.
\begin{defi}
 For every fixed $h>0$ and a bounded, differentiable function $f$ we define 
\begin{equation}\label{ZGLdiff}
{}_{h}^1\Delta_{c,\theta}f(x):=\sum\limits_{k=-\infty}^{\infty}\left(\omega_{k,1}\,h^{ik\tilde{c}-1}\sum\limits_{j=0}^{\infty}\binom{1-ik\tilde{c}}{j}(-1)^jf(x-jh)+\omega_{k,2}f'(x)\right),
\end{equation}
for every $x\in\mathbb{R}$, where $(\omega_{k,1})_{k\in\mathbb{Z}}$ and $(\omega_{k,2})_{k\in\mathbb{Z}}$ are given by (\ref{omega1Z}) and (\ref{omega2Z}). 
\end{defi}

With the same arguments as in Lemma \ref{abscon} and \ref{GLsfd} we get the following result.

\begin{lemma}\label{GLZsfd}
 Let $f\in L^1(\mathbb{R})$ be bounded and differentiable. Then for every $x\in\mathbb{R}$, the series in (\ref{ZGLdiff}) is absolutely convergent, real-valued and as $h\downarrow0$ we have
 \begin{equation}
  \widehat{{}_{h}^1\Delta_{c,\theta}f}(x)\to \left(\sum\limits_{k=-\infty}^{\infty} \omega_{k,1}(-ix)^{1-ik\tilde{c}}+\omega_{k,2}(-ix)\right)\widehat{f}(x).
 \end{equation}
\end{lemma}

\begin{defi}\label{DefDelta}
For every fixed $h>0$ and a bounded, differentiable function $f$ with $f'(y)=O(|y|^{-\beta})$ for some $\beta>0$ as $y\to-\infty$, we define 
\begin{equation*}
{}_{h}^2\Delta_{c,\theta}f(x):=-c_0\gamma f'(x)+c_0\sum\limits_{j=1}^{\infty} \frac{1}{j}\left(f'(x)1_{[0,h^{-1}]}(j)-f'(x-jh)\right)
\end{equation*}
for every $x\in\mathbb{R}$, where $\gamma\approx 0.5772$ is the Euler-Mascheroni constant. 
\end{defi}

\begin{remark}
 Due to the assumptions on $f'$ in Definition \ref{DefDelta}, ${}_{h}^2\Delta_{c,\theta}f(x)$ exists for every $x\in\mathbb{R}$. Further, for $h\in(0,1)$, we are able to give an alternative representation of the limit 
 as $h\downarrow 0$, namely
 \begin{align}\label{altlim}
  \lim\limits_{h\downarrow 0}{}_{h}^2\Delta_{c,\theta}f(x)&=\lim\limits_{h\downarrow 0}c_0h^{-1}\left(\frac{\partial^{h+1}}{\partial x^{h+1}}f(x)-f'(x)\right)
 \end{align}
for every $x\in\mathbb{R}$, where $\frac{\partial^{h+1}}{\partial x^{h+1}}f$ is the Caputo form of the usual fractional derivative of order $h+1\in(1,2)$. To see this, first note that
\begin{equation}\label{altlim1}\begin{split}
{}_{h}^2\Delta_{c,\theta}f(x)&=-c_0\gamma f'(x)+c_0\sum\limits_{j=1}^{\infty} \frac{1}{j}\left(f'(x)1_{[0,1]}(jh)-f'(x-jh)\right)\\
&\to -c_0\gamma f'(x)+c_0\int_{0}^{\infty} \frac{1}{y}\left(f'(x)1_{[0,1]}(y)-f'(x-y)\right)dy
\end{split}\end{equation}
due to the convergence of the Riemannian sums for $h\downarrow 0$. On the other hand, the right-hand side of (\ref{altlim}) equals
\begin{align*}
 &h^{-1}\left(\frac{\partial^{h+1}}{\partial x^{h+1}}f(x)-f'(x)\right)\\
 &\quad=\left(\frac{1}{\Gamma(1-h)}\int_0^{\infty}\left(f'(x)-f'(x-y)\right)y^{-h-1}dy-\int_1^{\infty} f'(x)y^{-h-1}dy\right)\\
 &\quad=\left(h^{-1}\left(\frac{1}{\Gamma(1-h)}-1\right)f'(x)\right.\\
 &\quad\qquad\left.+\frac{1}{\Gamma(1-h)}\int_0^{\infty}\left(f'(x)1_{[0,1]}(y)-f'(x-y)\right)y^{-h-1}dy\right).
\end{align*}
Since as $h\downarrow0$
\begin{align*}
 h^{-1}\left(\frac{1}{\Gamma(1-h)}-1\right)\to \left.\frac{d}{dx}\,\frac{1}{\Gamma(1-x)}\right|_{x=0}=\frac{\Gamma'(1)}{\Gamma(1)^2}=-\gamma
\end{align*}
with the Euler-Mascheroni constant $\gamma$, it follows that
\begin{align*}
 h^{-1}\left(\frac{\partial^{h+1}}{\partial x^{h+1}}f(x)-f'(x)\right)\to -\gamma\, f'(x)+\int_0^{\infty}\left(f'(x)1_{[0,1]}(y)-f'(x-y)\right)y^{-1}dy
\end{align*}
as $h\downarrow 0$ and this together with \eqref{altlim1} proves \eqref{altlim}.
\end{remark}

\begin{lemma}\label{GLZsfd1}
 Let $f\in L^1(\mathbb{R})$ be a bounded, differentiable function with $f'(y)=O(|y|^{-\beta})$ for some $\beta>0$ as $y\to-\infty$. Then for every $x\in\mathbb{R}$ as $h\downarrow 0$ we have
 \begin{equation}
  \widehat{{}_{h}^2\Delta_{c,\theta}f}(x)\to c_0(-ix)\log(-ix)\widehat{f}(x).
 \end{equation}
\end{lemma}

\begin{proof}
 Using equation (\ref{altlim}), we get
 \begin{align*}
  \lim\limits_{h\downarrow 0} \widehat{{}_{h}^2\Delta_{c,\theta}f}(x)&=\lim\limits_{h\downarrow 0}c_0h^{-1}\left( (-ix)^{h+1}-(-ix)\right)\widehat{f}(x)\\
  &=\lim\limits_{h\downarrow 0}c_0(-ix) \frac{(-ix)^h-1}{h}\widehat{f}(x)\to -c_0\,ix\log(-ix)\widehat{f}(x)
 \end{align*}
for every $x\in\mathbb{R}$.
\end{proof}

Combining Lemma \ref{GLZsfd} and \ref{GLZsfd1} with equation (\ref{FTsfZdalt}), Fourier inversion directly yields a {\it Gr\"unwald-Letnikov type formula} for the Zolotarev semi-fractional derivative.
\begin{theorem}\label{ZGL}
Let $f\in L^1(\mathbb{R})$ be a bounded, differentiable function with $f'(y)=O(|y|^{-\beta})$ for some $\beta>0$ as $y\to-\infty$ such that $f^{(2)},f^{(3)}$ exist and $f^{(3)}\in L^1(\mathbb{R})$. Then for almost every $x\in\mathbb{R}$ we have
\begin{align*}
\frac{\partial_{\mathcal{Z}}}{\partial_{c,\theta}x}\,f(x) & =\lim\limits_{h\downarrow 0}{}_{h}^1\Delta_{c,\theta}f(x)+{}_{h}^2\Delta_{c,\theta}f(x)\\
&=\lim_{h\downarrow0}\sum\limits_{k=-\infty}^{\infty}\left(\omega_{k,1}\,h^{ik\tilde{c}-1}\sum\limits_{j=0}^{\infty}\binom{1-ik\tilde{c}}{j}(-1)^jf(x-jh)+\omega_{k,2}f'(x)\right)\\
& \quad\qquad-c_0\gamma f'(x)+c_0\sum\limits_{j=1}^{\infty} \frac{1}{j}\left(f'(x)1_{[0,h^{-1}]}(j)-f'(x-jh)\right),
\end{align*}
where $(\omega_{k,1})_{k\in\mathbb{Z}}$ and $(\omega_{k,2})_{k\in\mathbb{Z}}$ are given by (\ref{omega1Z}) and (\ref{omega2Z}).
\end{theorem}

Note that for constant $\theta\equiv 2/\pi=c_0$ the Gr\"unwald-Letnikov type formula for the ordinary Zolotarev fractional derivative of order $\alpha=1$ from \cite{KLM} is given by
$$\frac{\partial_{\mathcal{Z}}}{\partial x}\,f(x)=-\frac2{\pi}\,\gamma f'(x)+\frac2{\pi}\lim_{h\downarrow0}\sum\limits_{j=1}^{\infty} \frac{1}{j}\left(f'(x)1_{[0,h^{-1}]}(j)-f'(x-jh)\right).$$

\begin{remark}
 Analogously to Remark \ref{nsf} and under the same conditions as in Theorem \ref{ZGL}, except that now $f'(y)=O(|y|^{-\beta})$ for some $\beta>0$ as $y\to+\infty$, we get a Gr\"unwald-Letnikov type formula for the negative Zolotarev semi-fractional derivative of order $\alpha=1$
\begin{align*}
\frac{\partial_{\mathcal{Z}}}{\partial_{c,\theta}(-x)}\,f(x) & =\lim_{h\downarrow0}\sum\limits_{k=-\infty}^{\infty}\left(\omega_{k,1}\,h^{ik\tilde{c}-1}\sum\limits_{j=0}^{\infty}\binom{1-ik\tilde{c}}{j}(-1)^jf(x+jh)-\omega_{k,2}f'(x)\right)\\
& \quad\qquad+c_0\gamma f'(x)+c_0\sum\limits_{j=1}^{\infty} \frac{1}{j}\left(f'(x+jh)-f'(x)1_{[0,h^{-1}]}(j)\right),
\end{align*}
\end{remark}

\begin{figure}[H]
\includegraphics[width=0.65\textwidth]{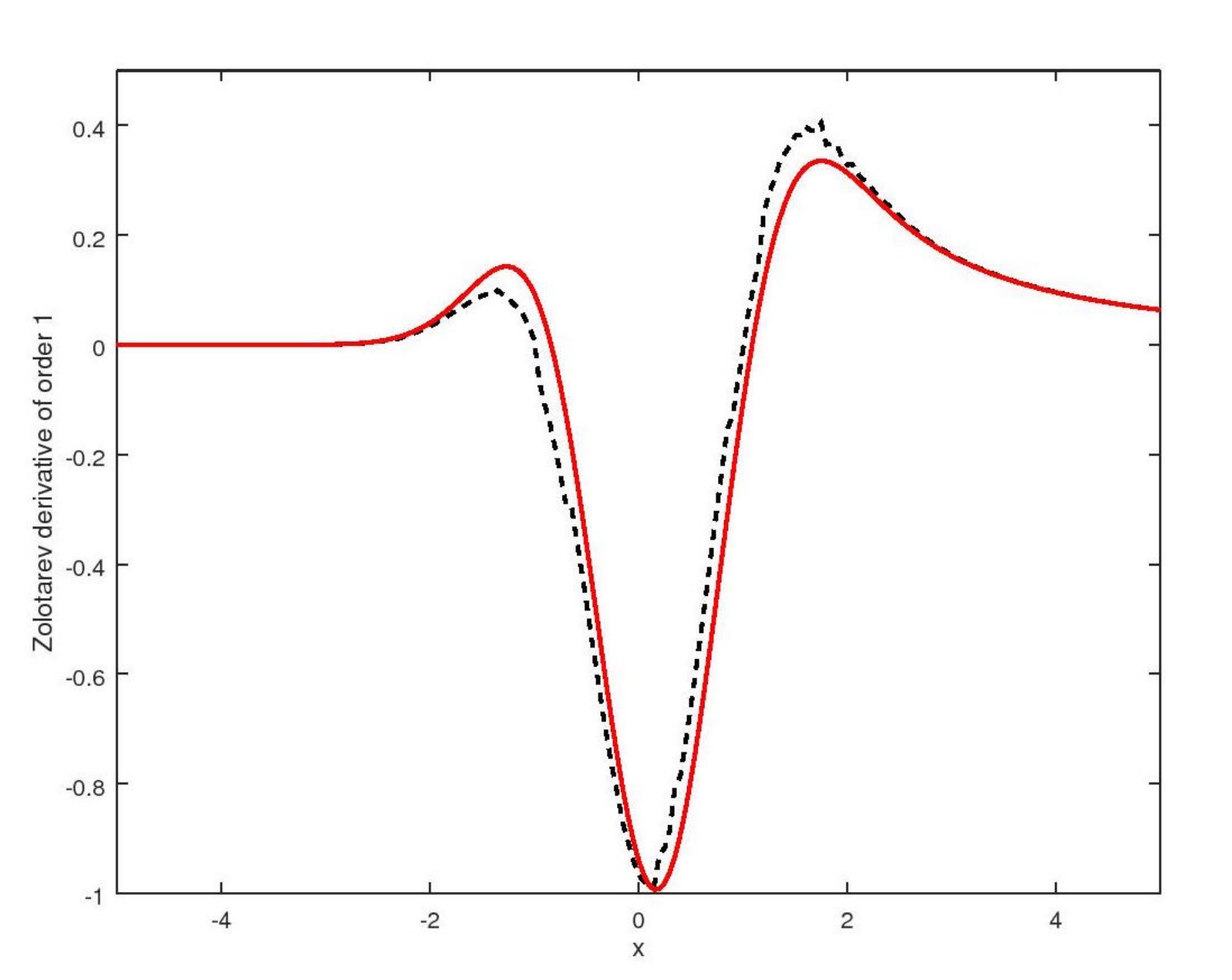}
\caption{Zolotarev semi-fractional derivative (dashed line) of order $\alpha=1$ in Definition \ref{generator5} on the interval $[-5,5]$ and its Gr\"unwald-Letnikov approximation \eqref{ZGLdiff} (solid line) from Example \ref{ex2}.}
\label{Fig3}
\end{figure} 

\begin{bsp}\label{ex2}
 We consider again $f(x)=\exp(-x^2)$ and $\theta(x)=\frac{2\sin(x)}{3\pi}+\frac{2}{\pi}$ such that eliminating the first term, we get back the ordinary Zolotarev fractional derivative of order $\alpha=1$. Then we are able to approximate the Zolotarev semi-fractional derivative of order $\alpha=1$ with our Gr\"unwald-Letnikov approach in Theorem \ref{ZGL}. 
 The result is shown in Figure \ref{Fig3}, where the approximation with step size $h=0.05$ as well as the numerically evaluated Caputo form as described in Example \ref{ex1} are plotted in the interval $[-5,5]$. Note that the dashed line shows some numerical anomaly which may be caused by the cosine integral part $\frac{\cos y}{y}$ together with the strong fluctuations of $\theta(\log y)$ near $y=0$ in the Caputo form of the Zolotarev semi-fractional derivative.
\end{bsp}

\section{Semi-fractional diffusion equations}

We are now able to deal with semi-fractional diffusion equations. In particular, for given $c>1$, $\alpha\in(0,2)$ and corresponding admissable functions $\theta_1,\theta_2$ we aim to give a solution of the equation
\begin{equation}\label{semifracDE}
\frac{\partial}{\partial t}\,p(x,t)=-v\,\frac{\partial}{\partial x}\,p(x,t)+D_1\,\frac{\partial^{\alpha}}{\partial_{c,\theta_1}x^{\alpha}}\,p(x,t)+D_2\,\frac{\partial^{\alpha}}{\partial_{c,\theta_2}(-x)^{\alpha}}\,p(x,t)
\end{equation}
for constants $v,D_1,D_2\in\rr$ with $D_1,D_2\leq0$ if $\alpha\in(0,1)$ or $D_1,D_2\geq0$ if $\alpha\in[1,2)$ and $D_1+D_2\not=0$ in both cases. In case $\alpha=1$ the semi-fractional derivatives are given by their Zolotarev form from Section 2.2. Note that in the symmetric case $D_1=D_2$ and $\theta:=\theta_1=\theta_2$ we may summarize 
$$\frac{\partial^{\alpha}p}{\partial_{c,\theta}|x|^{\alpha}}:=\frac{\partial^{\alpha}p}{\partial_{c,\theta}x^{\alpha}}+\frac{\partial^{\alpha}p}{\partial_{c,\theta}(-x)^{\alpha}}$$
to a {\it semi-fractional Laplacian} which can also be considered as a {\it semi-fractional Riesz derivative}; see \cite{KST,SKM} for their fractional counterparts. Let $\nu$ be the semistable distribution with L\'evy measure $\phi$ given by 
\begin{equation}\label{dif1}
\phi(r,\infty)=r^{-\alpha}|D_1|\theta_1(\log r)\quad\text{ and }\quad
 \phi(-\infty,-r)=r^{-\alpha}|D_2|\theta_2(\log r)
\end{equation}
and define
\begin{equation}\label{a}
a:=v+\begin{cases}\displaystyle\int_{\rr\setminus\{0\}}\frac{y}{1+y^2}\,d\phi(y) & \text{ if }\alpha\in(0,1),\\
\displaystyle\int_{\rr\setminus\{0\}}\left(\frac{y}{1+y^2}-\sin y\right)\,d\phi(y) & \text{ if }\alpha=1,\\
\displaystyle\int_{\rr\setminus\{0\}}\left(\frac{y}{1+y^2}-y\right)\,d\phi(y) & \text{ if }\alpha\in(1,2).\end{cases}
\end{equation}
as the drift coefficient in \eqref{logchar}. Note that the infinitely divisible distribution $\nu$ generates a continuous convolution semigroup $(\nu^{\ast t})_{t\geq0}$ representing the family of one-dimensional marginal distributions of a semistable L\'evy process $(X_t)_{t\geq0}$ with generator $L$ from \eqref{generator}. Hence the semi-fractional diffusion equation \eqref{semifracDE} is the corresponding abstract Cauchy problem for this semistable generator and the problem \eqref{semifracDE} is well-posed.
\begin{theorem}
Let $(X_t)_{t\geq0}$ be the semistable L\'evy process given uniquely in law by the semistable distribution $\nu$ of $X_1$ with L\'evy measure \eqref{dif1} and drift \eqref{a}. Then $X_t$ has a continuous Lebesgue density $x\mapsto p(x,t)$  for every $t>0$ and these densities are a solution to the semi-fractional diffusion equation \eqref{semifracDE} if $D_1,D_2\leq0$ in case $\alpha\in(0,1)$ or $D_1,D_2\geq0$ in case $\alpha\in[1,2)$.
\end{theorem}
\begin{proof}
First note that for every semistable L\'evy process $(X_t)_{t\geq 0}$ a continuous Lebesgue density of $X_t$ exists for $t>0$ and is in fact a function belonging to $C^{\infty}(\rr)$; see \cite{MMMHPS,Sato} for details. 
Denoting $\partial_{\mathcal Z}=\partial^1$ in case $\alpha=1$ to simplifiy notation, for the Fourier transform of the density we obtain with the log-characteristic function $\psi$ as in \eqref{logchar} and the generator $L$ of the L\'evy process
\begin{align*}
\frac{\partial}{\partial t}\,\widehat p(k,t) & =\frac{\partial}{\partial t}\,\Exp[\exp(ikX_t)]=\frac{\partial}{\partial t}\,\exp(t\psi(k))=\psi(k)\widehat p(k,t)=\widehat{Lp}(k,t)\\
&=\begin{cases}\displaystyle ivk\,\widehat p(k,t)-|D_1|\,\widehat{\frac{\partial^{\alpha}p}{\partial_{c,\theta_1}x^{\alpha}}}\,(k,t)-|D_2|\,\widehat{\frac{\partial^{\alpha}p}{\partial_{c,\theta_2}(-x)^{\alpha}}}\,(k,t) & \text{ if }\alpha\in(0,1)\\[3ex]
\displaystyle ivk\,\widehat p(k,t)+|D_1|\,\widehat{\frac{\partial^{\alpha}p}{\partial_{c,\theta_1}x^{\alpha}}}\,(k,t)+|D_2|\,\widehat{\frac{\partial^{\alpha}p}{\partial_{c,\theta_2}(-x)^{\alpha}}}\,(k,t) & \text{ if }\alpha\in[1,2)
\end{cases}\\[2ex]
&=ivk\,\widehat p(k,t)+D_1\,\widehat{\frac{\partial^{\alpha}p}{\partial_{c,\theta_1}x^{\alpha}}}\,(k,t)+D_2\,\widehat{\frac{\partial^{\alpha}p}{\partial_{c,\theta_2}(-x)^{\alpha}}}\,(k,t)
\end{align*}
where the last equalities follow according to Definitions \ref{generator1}--\ref{generator4} in case $\alpha\not=1$ and Definitions \ref{generator5}--\ref{generator6} in case $\alpha=1$ together with the sign restrictions of the constants $D_1,D_2$. Since the densities $x\mapsto p(x,t)$ belong to $C_0(\rr)\cap L^1(\rr)$ for all $t>0$, applying Fourier inversion directly leads to \eqref{semifracDE}.
\end{proof}
We now turn to numerical solutions of the semi-fractional diffusion equation \eqref{semifracDE} on a rectangle $x\in[-b,b]$, $t\in[T_1,T_2]$ for some $b>0$, $T_1,T_2>0$, assuming for simplicity $v=0$ for the drift part. Given $c>1$ and $\alpha\in(0,1)$ we choose a smooth admissable function $\theta=\theta_1=\theta_2$ and $D_1,D_2\leq 0$ with $D_1+D_2\neq 0$. To calculate the solution $p$ numerically, we approximate the left-hand side of \eqref{semifracDE} by a classical finite difference of order one and the (negative) semi-fractional derivative of order $\alpha$ on the right-hand side of \eqref{semifracDE} by our Gr\"unwald-Letnikov formula. Thereby, we choose fixed step sizes $\Delta t:=0.01$ in time and $h:=0.01$ in space. In order to get a good approximation for the semi-fractional derivatives, we approximate the solution on a larger interval in space, such that we consider $50$ neighboring points to the left and to the right of every point of interest $x\in[-b,b]$ for the calculation. To start our numerical algorithm for the calculation of a semistable density $p$, the initial condition $p(x,0)=\delta_x$ given by the Dirac delta function is not appropriate. Therefore, we choose 
$p(x,T_1)\approx p_1(x,T_1)$ for $T_1=0.01$ as the solution of the corresponding fractional diffusion equation; i.e.
\begin{equation*}
 \frac{\partial}{\partial t}p_1(x,t)=D_1\frac{\partial^{\alpha}}{\partial x^{\alpha}}p_1(x,t)+D_2\frac{\partial^{\alpha}}{\partial(-x)^{\alpha}}p_1(x,t).
\end{equation*}
with initial condition $p_1(x,0)=\delta_x$. This is a stable density for which numerical algorithms exist.
 Since $p_1(x,T_1)$ is a function with a single sharp peak around zero, the influence of the log-periodic perturbations of the semi-fractional derivative (compared to the fractional derivative) should be very low. We calculated all starting values of $p_1$ with the function \textit{'dstable'} in \textit{R} (version 3.2.3). 
  \begin{figure}[H]
	\centering
	\includegraphics[width=0.49\textwidth]{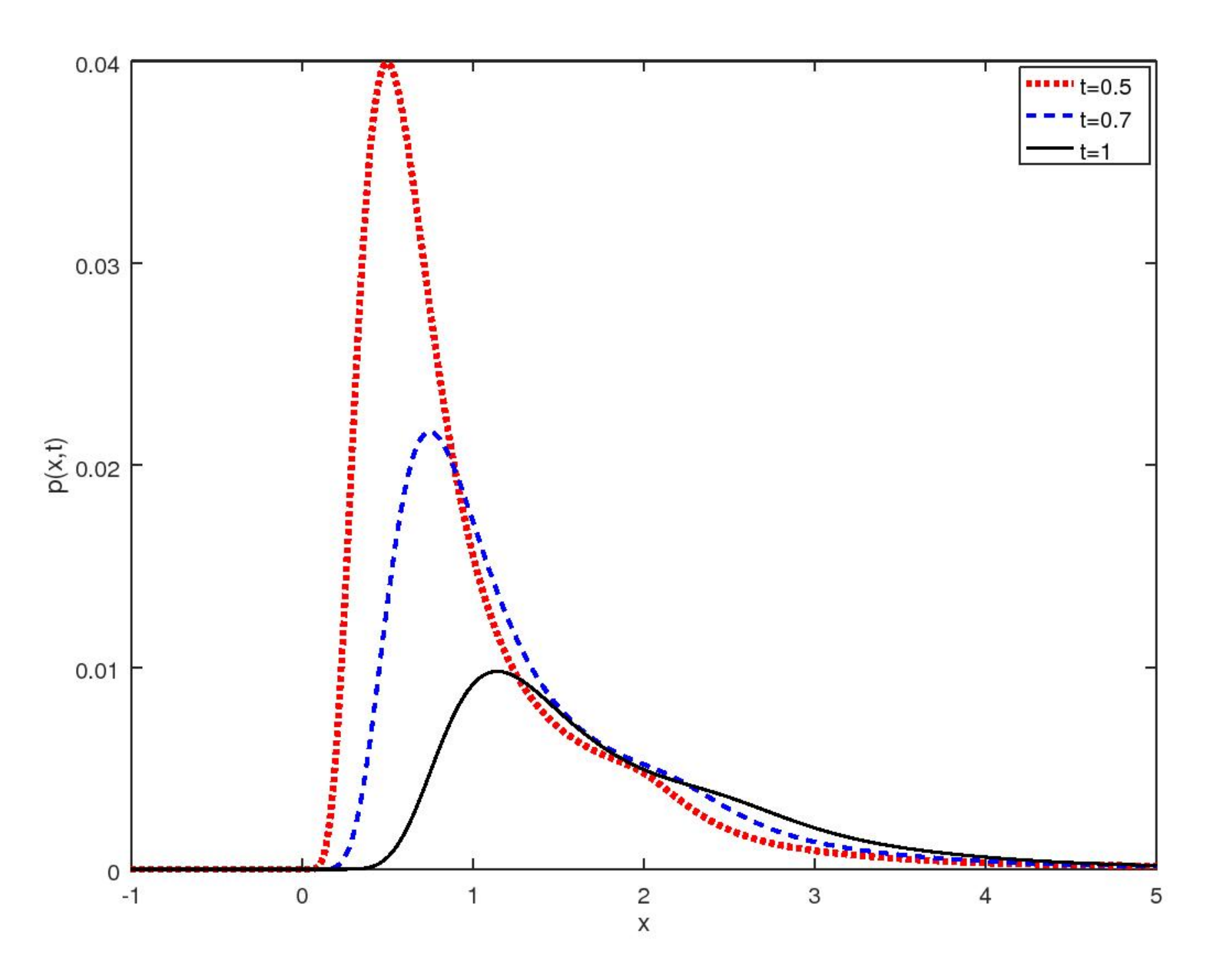}
	\includegraphics[width=0.47\textwidth]{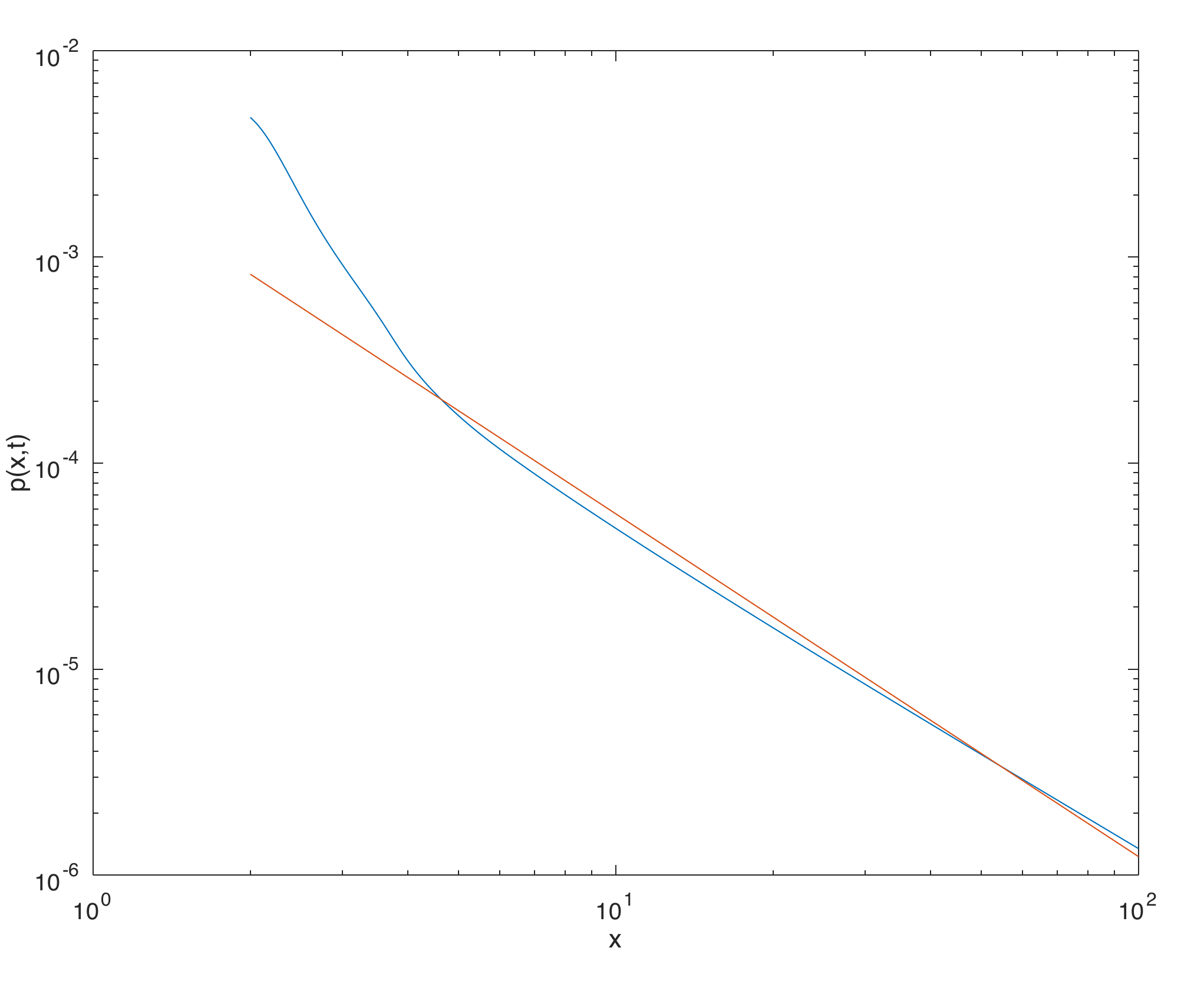}
	\caption{Left: Solution of the semi-fractional differential equation in Example \ref{51} at different times ($t=0.5$, $0.7$ and $1.0$). Right: log-log plot of the solution in Example \ref{51}.}
	\label{fig2}
\end{figure}
 \begin{bsp}\label{51}
Let $D_1=-1$, $D_2=0$, $b=5$ and $T_2=1$ such that \eqref{semifracDE} reduces to 
 \begin{align*}
 \frac{\partial}{\partial t}p(x,t)=-\frac{\partial^{\alpha}}{\partial_{c,\theta}x^{\alpha}}p(x,t) \quad\text{for all }x\in[-5,5],\,t\in[0,1].
 \end{align*}
 In addition, we choose $\alpha=0.5$, $c=e^{\pi}$ and $\theta(x)=\alpha\sin(x)+\Gamma(1-\alpha)$. Then $\theta$ is a smooth admissable function with respect to $c$ and $\alpha$. 
Following Remark 5.10 in \cite{MMMSik} the solution $p_1$ of the corresponding fractional equation at time $T_1=0.01$ (our starting point) is given by $S_{\alpha}(1,\sigma,0)$, where $\sigma=(T_1\cos(\frac{\pi\alpha}{2}))^{1/\alpha}$. Starting from $p_1$ 
We are now able to approximately calculate the solution of the semi-fractional diffusion equation. In Figure \ref{fig2} the result is given for different values of $t\in[T_1,T_2]$ and a log-log plot of the solution shows oscillations about a straight line which can also be seen in practical applications \cite{Sor}.
 \end{bsp}
\begin{figure}[H]
 \includegraphics[scale=0.35]{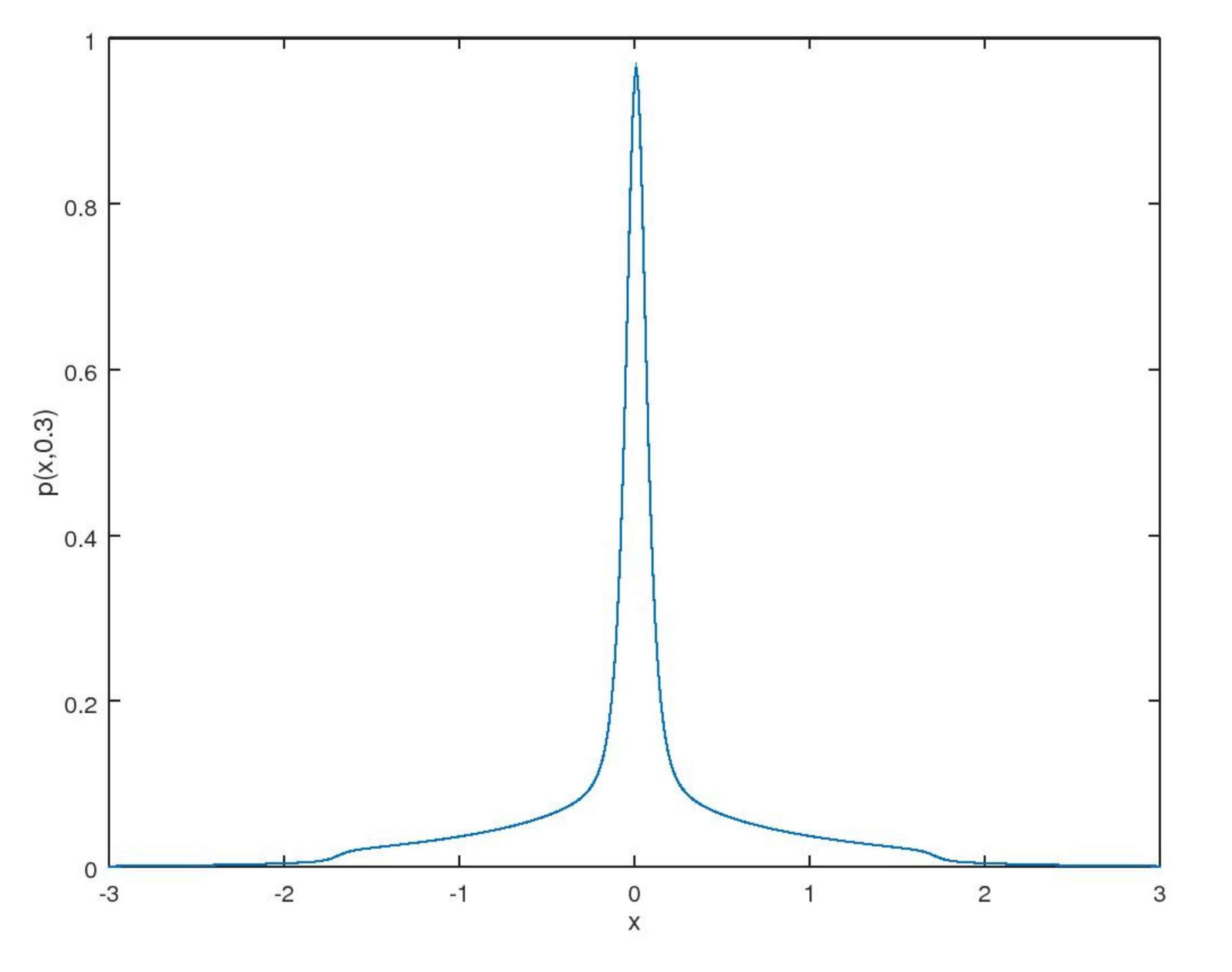}
 \includegraphics[scale=0.35]{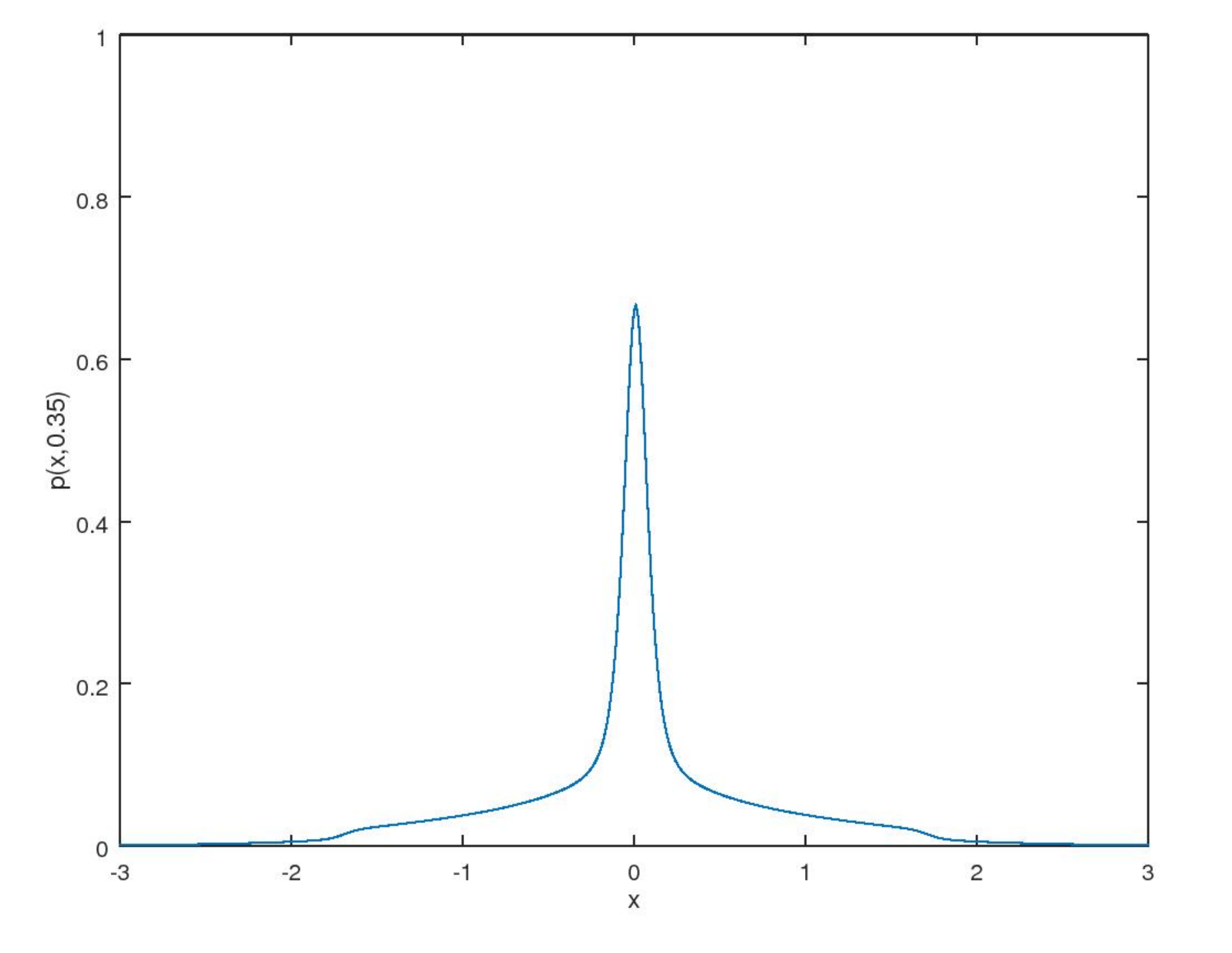}
 \includegraphics[scale=0.35]{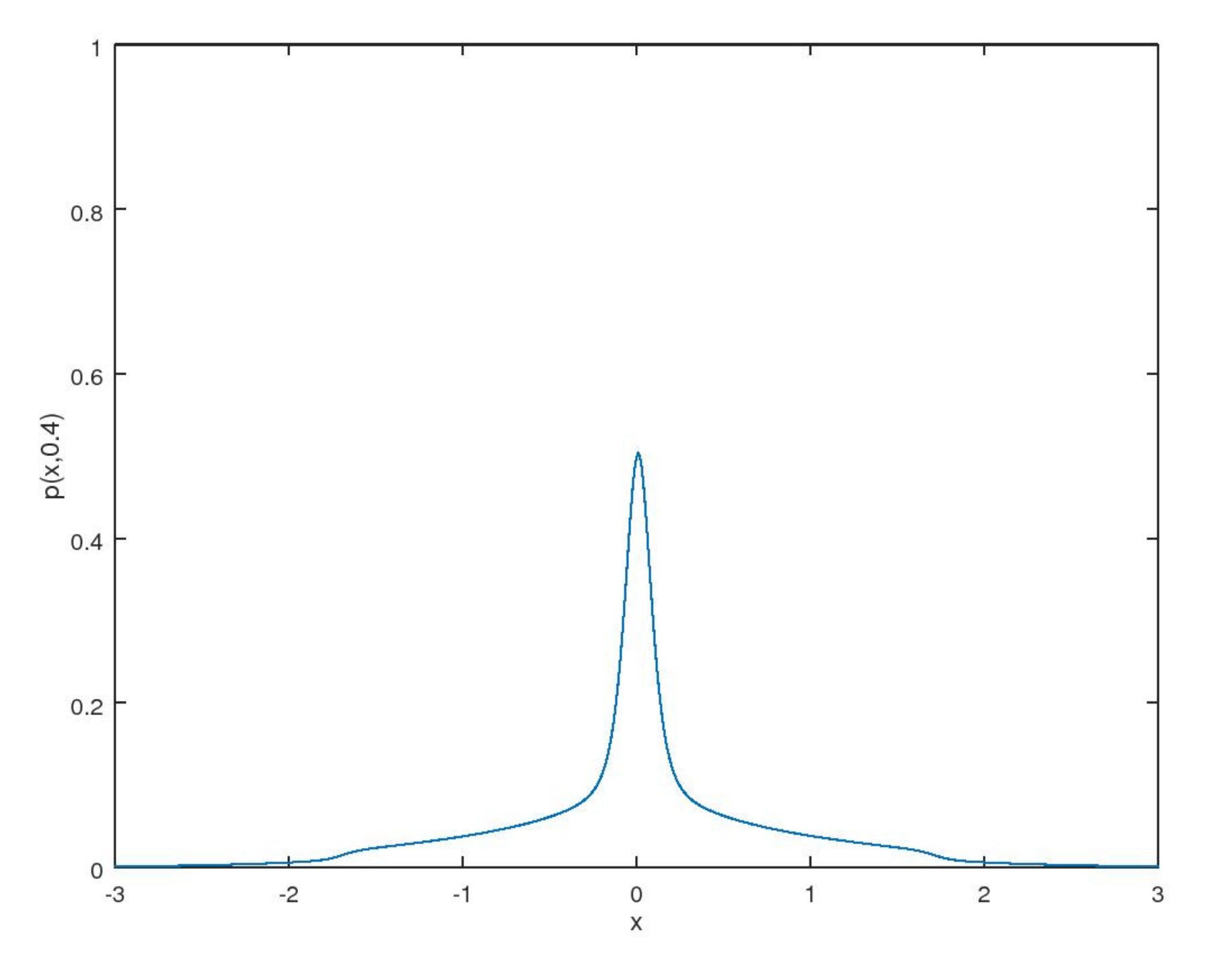}
 \includegraphics[scale=0.35]{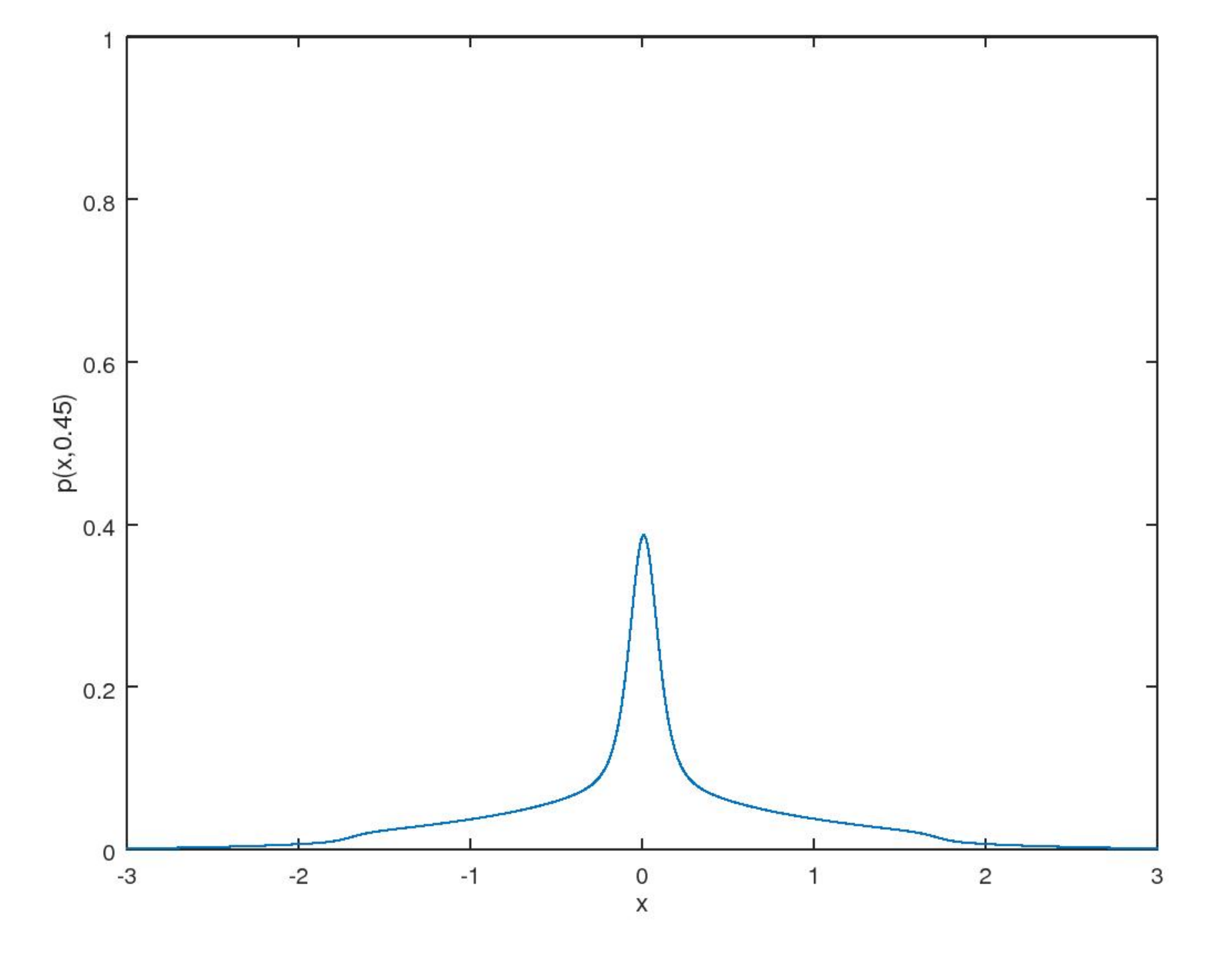}
  \caption{Solution of the semi-fractional differential equation in Example \ref{52} at different times ($t=0.3$, $0.35$, $0.4$ and $0.45$).}\label{fig4}
\end{figure}
 \begin{bsp}\label{52}
  Let $D_1=-0.5=D_2$, $b=5$ and $T_2=0.45$. Again, we choose $\alpha=0.5$ but this time we consider the function $\theta(x)=\alpha\cos(x)+\Gamma(1-\alpha)$ with $c=e^{\pi}$. Then $\theta$ is a smooth admissable function with
  respect to $\alpha$ and $c$. Our starting point is the corresponding stable density at $T_1=0.01$ with representation $S_{\alpha}(0,\sigma,0)$, where $\sigma$ is as in Example \ref{51}.
In Figure \ref{fig4}, the numerically calculated semistable density is shown for different values of $t\in[T_1,T_2]$.    
Note that the pictures indicate the appearance of more than one change between convexity and concavity in each of the tails of the semistable densities, which for stable distributions cannot happen. This is also apparent from numerical calculations of one-sided semistable densities by Laplace inversion techniques in \cite{Cha}.
\end{bsp}

\bibliographystyle{plain}

\begin{thebibliography}{10}
	
\bibitem{Stirling}
Andrews, G.E.; Askey, R.; and Roy, R. (1999)
{\it Special Functions.}
Cambridge University Press, Cambridge.

\bibitem{Octave}
  Bateman, D.; Eaton, J.W.; Hauberg, S.; and Wehbring, R. (2016)
  GNU Octave version 4.2.0 manual: a high-level interactive language for
  numerical computations.\\
 {\tt www.gnu.org/software/octave/doc/interpreter/}
 
\bibitem{Integral}
Bateman, H. (1954)
{\it Tables of Integral Transforms,} Vol. 1.
McGraw-Hill, New York.

\bibitem{Cha}
Chaudhuri, R. (2014)
{\it Non-Gaussian Semi-Stable Distributions and Their Statistical Applications.}
Ph.D. Thesis, University of North Carolina, Chapel Hill.\\
{\tt cdr.lib.unc.edu/indexablecontent/uuid:a92e2348-bb06-4e0e-9a65-ba8112406df0}

\bibitem{Chavez} 
Chavez, A. (2000) 
A fractional diffusion equation to describe L\'evy flights. 
{\it Phys. Lett. A} {\bf 239} 13--16.

\bibitem{Flojolet}
Flajolet, P.; and Sedgewick, R. (2009)
{\it Analytic Combinatorics.}
Cambridge University Press, Cambridge.

\bibitem{Fourier1}
Folland, G.B. (1992)
{\it Fourier Analysis and Its Applications.}
Wadsworth \& Brooks/Cole, London.

\bibitem{binom}
Hazewinkel, M. ed. (1988)
{\it Encyclopaedia of Mathematics,} Vol. 1.
Reidel, Kluwer, Dordrecht.

\bibitem{HPBA}
Huillet, T.; Porzio, A.; and Ben Alaya, M. (2001)
On L\'evy stable and semistable distributions.
{\it Fractals} {\bf 9} 347--364.

\bibitem{KLM}
Kelly, J.F.; Li, C.G.; and Meerschaert, M.M. (2018)
Anomalous diffusion with ballistic scaling: A new fractional derivative.
{\it J. Comp. Appl. Math.} {\bf 339} 161--178.

\bibitem{KST}
Kilbas, A.A.; Srivastava, H.M.; and Trujillo, J.J. (2006)
{\it Theory and Applications of Fractional Differential Equations.}
North-Holland Mathematical Studies {\bf 204}, Elsevier, Amsterdam.

\bibitem{ML}
Martin-L\"of, A. (1985)
A limit theorem which clarifies the ``Petersburg paradox''. 
{\it J. Appl. Probab.} {\bf 22} 634--643.

\bibitem{MMMHPS}
Meerschaert, M.M.; and Scheffler, H.P. (2001)
{\it Limit Distributions for Sums of Independent Random Vectors.}
Wiley, New York.

\bibitem{neben}
Meerschaert, M.M.; and Scheffler, H.P. (2002)
Semistable L\'evy motion.
{\it Fractional Calculus and Applied Analysis} {\bf 5} 27--54.

\bibitem{MMMSik}
Meerschaert, M.M.; and Sikorskii, A. (2012)
{\it Stochastic Models for Fractional Calculus.}
De Gruyter, Berlin.

\bibitem{MMMTad}
Meerschaert, M.M.; and Tadjeran, C. (2006) 
Finite difference approximations for two-sided space-fractional partial differential equations. 
{\it Appl. Numerical Math.} {\bf 56} 80--90.

\bibitem{MK} 
Metzler, R.; and Klafter, J. (2000) 
The random walk's guide to anomalous diffusion: A fractional dynamics approach. 
{\it Phys. Rep.} {\bf 339} 1--77.

\bibitem{SKM}
Samko, S.G.; Kilbas, A.A.; and Marichev, O.I. (1993) 
{\it Fractional Integrals and Derivatives.}
Gordon and Breach, London.

\bibitem{Sato}
Sato, K. (1999)
{\it L\'evy Processes and Infinitely Divisible Distributions.}
Cambridge University Press, Cambridge.

\bibitem{Sor} 
Sornette, D. (1998) 
Discrete-scale invariance and complex dimensions. 
{\it Physics Reports} {\bf 297} 239--270.

\bibitem{Sor2}
Sornette, D. (2017)
{\it Why Stock Markets Crash: Critical Phenomena in Complex Financial Systems.}
Princeton University Press, Princeton.

\end{thebibliography}

\end{document}